\documentclass{daj}
\usepackage{amssymb}
\usepackage{mathtools}
\usepackage{amsthm}

\newcommand{\restr}{|}
\newcommand{\boC}{\mathcal{C}}
\newcommand{\boD}{\mathcal{D}}
\newcommand{\boL}{\mathcal{L}}
\newcommand{\boM}{\mathcal{M}}
\newcommand{\boN}{\mathcal{N}}
\newcommand{\boP}{\mathcal{P}}
\newcommand{\boQ}{\mathcal{Q}}
\newcommand{\Hbb}{\mathbb{H}}
\newcommand{\Ebb}{\mathbb{E}}
\newcommand{\Z}{\mathbb{Z}}
\newcommand{\N}{\mathbb{N}}
\newcommand{\R}{\mathbb{R}}
\newcommand{\C}{\mathbb{C}}
\newcommand{\dd}{\mathop{}\!\mathrm{d}}
\newcommand{\ic}{\mathbf{i}}
\DeclarePairedDelimiter{\abs}{\lvert}{\rvert}
\DeclarePairedDelimiter{\norm}{\lVert}{\rVert}
\DeclarePairedDelimiter{\paren}{(}{)}
\newcommand{\st}{\::\:}
\newcommand{\gromprod}[3]{(#1 | #2)_{#3}}

\newcommand{\bfx}{\mathbf{x}}
\newcommand{\bfy}{\mathbf{y}}
\newcommand{\bfz}{\mathbf{z}}

\DeclareMathOperator{\Id}{Id}
\DeclareMathOperator{\supp}{supp}

\renewcommand{\epsilon}{\varepsilon}

\renewcommand{\phi}{\varphi}
\renewcommand{\leq}{\leqslant}
\renewcommand{\geq}{\geqslant}

\newtheorem{thm}{Theorem}[section]
\newtheorem{prop}[thm]{Proposition}

\newtheorem{lem}[thm]{Lemma}
\newtheorem{cor}[thm]{Corollary}
\newtheorem*{prop*}{Proposition}

\theoremstyle{definition}

\newtheorem*{ex*}{Example}
\newtheorem{rmk}[thm]{Remark}

\numberwithin{equation}{section}

\dajAUTHORdetails{%
  title = {Analyticity of the Entropy and the Escape Rate of Random Walks in Hyperbolic Groups}, 
  author = {S\'ebastien Gou\"ezel},
  plaintextauthor = {Sebastien Gouezel},
  runningtitle = {Analyticity of the Entropy and the Escape Rate},
}   

\dajEDITORdetails{%
   year={2017},
   number={7},
   received={29 September 2015},   
   revised={19 November 2016},    
   published={9 May 2017},  
   doi={10.19086/da.1639},       
}   

\begin{document}

\begin{frontmatter}[classification=text]


\author[sg]{S\'ebastien Gou\"ezel}

\begin{abstract}
We consider random walks on a nonelementary hyperbolic group endowed with a
word distance. To a probability measure on the group are associated two
numerical quantities, the rate of escape and the entropy. On the set of
admissible probability measures whose support is contained in a given
finite set, we show that both quantities depend in an analytic way on the
probability measure. Our spectral techniques also give a new proof of the
central limit theorem, and imply that the corresponding variance is
analytic.
\end{abstract}
\end{frontmatter}


\section{Introduction}

Let $\Gamma$ be a finitely generated group. Let $\mu$ be a finitely supported
probability measure on $\Gamma$, whose support generates $\Gamma$ as a
semigroup (we say that $\mu$ is \emph{admissible}). It defines a random walk
on $\Gamma$, by $Z_n=g_1\dotsm g_n$ where $g_1, g_2,\dotsc$ is a sequence of
$\Gamma$-valued i.i.d.\ random variables with distribution $\mu$. By
definition, $Z_n$ is distributed according to the convolution product
$\mu^{*n}$. Equivalently, $Z_n$ is a Markov chain on $\Gamma$, starting from
the identity $e$ of $\Gamma$, whose transition probabilities are $p(x\to
y)=\mu(x^{-1}y)$. These probabilities are invariant under left
multiplication, hence the Markov chain is homogeneous.

The behavior of the random walk is usually strongly related to the geometric
properties of the group. One can associate several numerical quantities to
the random walk, including:
\begin{itemize}
\item The entropy, defined by
\begin{equation*}
  h(\mu) = -\lim_{n \to \infty} \log \mu^{*n}(\{Z_n\})/n.
\end{equation*}
  The almost sure convergence of this quantity follows from Kingman's
    subadditive theorem. Essentially, the walk at time $n$ visits $e^{hn}$
    points.
\item The escape rate, or drift, defined by
\begin{equation*}
  \ell(\mu) = \lim_{n \to \infty} d(e,Z_n)/n,
\end{equation*}
where $d$ is a fixed (proper, left-invariant) distance on $\Gamma$. Again,
the convergence follows from Kingman's theorem.
\end{itemize}
These quantities are always non-zero in non-amenable groups, and always zero
in nilpotent groups for symmetric walks. Our focus in this article will be on
the former category, and especially on the subclass of hyperbolic groups.

In hyperbolic groups, Erschler and Kaimanovich have proved
in~\cite{erschler_kaim} that both the entropy and the rate of escape depend
continuously on the measure. If one concentrates on measures with a given
finite support, then the parameter space becomes a subset of $\R^d$ for some
$d$, and one can investigate further regularity properties. Our goal in this
article is to show that the entropy, and the rate of escape for a word
distance, are analytic.

Let $\Gamma$ be a nonelementary hyperbolic group, endowed with a word
distance $d$. Let $F$ be a finite subset of $\Gamma$. Denote by $\boP_1^+(F)$
the set of admissible probability measures supported by $F$: this is a subset
of the finite-dimensional space $\boP(F)=\{ \mu:F \to \C\}$. Our main theorem
is the following.

\begin{thm}
\label{thm:main} In this setting, the functions $\mu\mapsto \ell(\mu)$ and
$\mu\mapsto h(\mu)$ are analytic on $\boP_1^+(F)$. More precisely, there
exists an open subset of $\boP(F)$ containing $\boP_1^+(F)$ on which these
two functions extend to analytic functions.
\end{thm}

The random walk converges almost surely to a point of the Gromov boundary of
$\Gamma$. The distribution of the limit is called the exit measure of the
random walk. Its Hausdorff dimension is proportional to
$h/\ell$~\cite[Theorem 4.11]{haissinsky}. Thus, it follows from the previous
theorem that it also depends analytically on $\mu$.

This theorem extends several partial results in the literature. Here are some
previous results in this direction, we refer the reader
to~\cite{gilch_ledrappier} for more statements.
\begin{itemize}
\item For nearest neighbor random walks on the free group, one can obtain
    explicit formulas for $h(\mu)$ and $\ell(\mu)$, thereby proving
    analyticity. Similar results hold in free products,
    see~\cite{gilch_free_products, gilch_free_product_entropy}.
\item In a free group, for measures with a given finite support, both the
    rate of escape and the entropy are analytic, see
    respectively~\cite{gilch_free_group}
    and~\cite{ledrappier_entropy_free_group}.
\item In a general hyperbolic group, but for a restricted class of
    distances, Gilch and Ledrappier prove in~\cite{gilch_ledrappier} that
    the rate of escape is analytic. The condition on the distance is that
    its Busemann boundary should coincide with the Gromov boundary (these
    terms are defined in Paragraph~\ref{subsec:hyperbolic}). This is for
    instance the case if $\Gamma$ acts cocompactly by isometries on a
    hyperbolic space $\Hbb^n$ and the distance is given by $d(g, g') =
    d_{\Hbb^n}(g\cdot O, g'\cdot O)$ where $O$ is a suitable basepoint. On
    the other hand, if $\Gamma$ is not virtually free, this is never the
    case for a word distance as the Busemann boundary is totally
    disconnected, contrary to the Gromov boundary (see
    Paragraph~\ref{subsec:hyperbolic}).
\item In surface groups,~\cite{HMM_renewal_surface} shows that the rate of
    escape is analytic.
\item In a general hyperbolic group, Ledrappier proves
    in~\cite{ledrappier_entropy_hyperbolic} that both the rate of escape
    and the entropy are Lipschitz continuous.
\item This is improved to $C^1$ by Mathieu
    in~\cite{mathieu_entropy_hyperbolic}.
\end{itemize}

Our approach to prove Theorem~\ref{thm:main} is based on the strategy of
Ledrappier in~\cite{ledrappier_entropy_hyperbolic}: We express the rate of
escape and the entropy as integrals of objects living on a suitable boundary
of the group, the Busemann boundary $\partial_B \Gamma$. What we have to show
is that these objects (the stationary measure and the Martin kernel) depend
analytically on the measure $\mu$.

More precisely, let $\nu_\mu$ be a stationary measure for $\mu$ on
$\partial_B \Gamma$. Let $K_\mu$ be the Martin kernel for $\mu$, and
$K_\mu^B$ its lift to $\partial_B \Gamma$. Let $c_B$ be the Busemann cocycle
on $\partial_B \Gamma$. (All these terms will be defined in
Paragraph~\ref{subsec:hyperbolic}.) Then the following formulas are classical
(see for instance~\cite[Proposition 2.2]{gouezel_matheus_maucourant_2}
for~\eqref{eq:ell}, and~\cite[Theorem~3.1]{kaim_vershik} or the comments
before~\eqref{eq:h_integral} for~\eqref{eq:h}):
\begin{equation}
\label{eq:ell}
  \ell(\mu) = \sum_{g\in \Gamma} \mu(g) \int_{\partial_B \Gamma} c_B(g, \xi) \dd\nu_\mu(\xi)
\end{equation}
and
\begin{equation}
\label{eq:h}
  h(\mu) = -\sum_{g\in \Gamma} \mu(g) \int_{\partial_B \Gamma} \log K_\mu^B(g^{-1},\xi)\dd\nu_\mu(\xi).
\end{equation}

There are two main difficulties to prove analyticity statements using these
formulas.

First, to use~\eqref{eq:ell}, we need to know that $\mu \mapsto \nu_\mu$ is
analytic in some sense. This is well known if the action of $\Gamma$ on
$\partial_B \Gamma$ has contraction properties, for instance if $\partial_B
\Gamma$ coincides with the Gromov boundary $\partial \Gamma$. However, if
$\Gamma$ is not virtually free, the projection $\pi_B : \partial_B \Gamma \to
\partial \Gamma$ is not one-to-one (see Paragraph~\ref{subsec:hyperbolic}),
and points in the same fiber of $\pi_B$ do not feel contraction. As a
consequence, we do not know if there is a unique stationary measure on
$\partial_B\Gamma$. To work around this problem, we work on a subset
$\partial'_B \Gamma$ of $\partial_B \Gamma$ which is minimal for the
$\Gamma$-action, and show using a non-constructive spectral argument that
there is some form of contraction there (more precisely, a spectral gap for
the Markov operator associated to the walk, on a suitable function space,
with a simple dominating eigenvalue). Hence, there is a unique stationary
measure on $\partial'_B \Gamma$, depending analytically on $\mu$. Together
with~\eqref{eq:ell}, this shows that $\ell(\mu)$ depends analytically on
$\mu$.

Second, in~\eqref{eq:h}, the cocycle $\log K_\mu^B(g^{-1},\xi)$ also depends
on $\mu$. This difficulty is serious, making the entropy harder to study in
general than the escape rate. To get the analyticity of the entropy, we need
to show that $\mu\mapsto K_\mu^B$ is analytic in some sense. Ledrappier
obtains in~\cite{ledrappier_entropy_hyperbolic} the kernel
$K_\mu^B(\cdot,\xi)$ for each $\xi$ by applying a sequence of contracting
operators depending on $\xi$. Choosing carefully the number of contraction
steps, he deduces a Lipschitz control on $K_\mu^B$, and thus that the entropy
is Lipschitz. Instead, we will exhibit the whole kernel $(x,\xi)\mapsto
K_\mu^B(x,\xi)$ as the unique fixed point of a non-linear operator, depending
analytically on $\mu$. The analyticity of $\mu \mapsto K_\mu^B$ then follows
from a suitable application of the implicit function theorem.

This approach is reminiscent of the study of unstable foliations in
hyperbolic dynamics. The usual strategy is to apply a map called the graph
transform, enjoying contraction properties, to obtain the unstable leaf at a
point. Unstable leaves at nearby points are then compared by iterating the
graph transform a finite but carefully chosen number of times. Another
approach, advocated by Hirsch, Pugh and Shub in~\cite{hirsch_pugh_shub}, is
to see the whole family of unstable manifolds as the fixed point of a single
operator. The proper setting is more complicated to develop, but once this is
done it is extremely powerful. We follow essentially a similar strategy.

\begin{rmk}
The kernel $K_\mu(x,\xi)$ is initially defined on the (geometric) Gromov
boundary. However, to construct an operator of which it is a fixed point, it
is convenient to have an underlying combinatorial structure. This is more
efficiently done using the Busemann boundary. Therefore, although the
statement on analyticity on the entropy does not depend on a distance choice,
its proof relies on the tools we develop to show that the escape rate is
analytic for the word distance.
\end{rmk}

\begin{rmk}
The escape rate $\ell(\mu)$ depends on the measure $\mu$, but also on the
choice of a distance on $\Gamma$. Theorem~\ref{thm:main} is formulated for a
word distance. Other natural distances can be constructed as follows.
Consider a free cocompact action of $\Gamma$ by isometries on a
Gromov-hyperbolic space $X$, take a basepoint $O \in X$, and define $d(g, g')
= d_X(g \cdot O, g' \cdot O)$. (The case of a word distance corresponds to
the natural action of $\Gamma$ on its Cayley graph with the graph distance).
We do not know if the escape rate $\ell(\mu)$ depends analytically on $\mu$
in this generality. Indeed, our proofs rely on some contraction properties of
the dynamics on the Busemann boundary, for some suitable distance, which are
not clear in this generality.

For the word distance, we prove below that these requirements are met (the
distance on the Busemann boundary is defined in~\eqref{eq:def_d}, the
contraction properties are proved in Lemma~\ref{lem:LY}). The situation is
also well behaved if one considers an action of $\Gamma$ by isometries on
$X=\mathbb{H}^k$ (or a more general $CAT(-1)$ space) and defines a distance
$d$ on $\Gamma$ using the $X$ distance on an orbit as above. Indeed, in this
case, the Busemann boundary can be identified with a subset of the Gromov
boundary of $X$. One thus gets a distance with good contracting properties on
the Busemann boundary, from which one can deduce that the drift is analytic
following (simplified versions of) the proofs below. This case is in fact
easier than the case of the word distance, and was already known, see
Corollary~4.2 in~\cite{gilch_ledrappier}.
\end{rmk}

\begin{rmk}
Theorem~\ref{thm:main} only proves analyticity of the escape rate and the
entropy on the set of admissible measures. One may hope that these quantities
are analytic on a larger set of measures, for instance those that generate a
nonelementary subgroup. However, there are considerable difficulties to
extend the proofs to this more general situation. For the escape rate, the
problem is in Proposition~\ref{prop:spectral_description}: we consider there
the dominating eigenvalue of an operator $L_\mu$ associated to $\mu$. If
$\mu$ is not admissible, this eigenvalue can have a nontrivial multiplicity,
and perturbations of non-simple eigenvalues can be non-analytic. The
situation is even worse for the analyticity of the entropy: there, we make
use of the whole machinery of Ancona inequalities (see for instance
Proposition~\ref{prop:ineq_ancona_strong}), which are only known for
admissible measures.
\end{rmk}

Once the spectral gap is available, standard techniques due to Le
Page~\cite{lepage_trick} imply several results on the behavior of the random
walk, including the central limit theorem, the law of the iterated logarithm
and large deviations estimates. As an illustration, we prove the following
statement:
\begin{thm}
\label{thm:CLT} Under the assumptions of Theorem~\ref{thm:main}, there exists
for all $\mu\in \boP^+_1(F)$ a real number $\sigma^2(\mu)> 0$ such that
$(d(e,Z_n)-n\ell(\mu))/\sqrt{n}$ converges in distribution to $\boN(0,
\sigma^2(\mu))$ when $n$ tends to infinity. Moreover, $\mu\mapsto
\sigma^2(\mu)$ is analytic on $\boP^+_1(F)$.
\end{thm}

The central limit theorem is already known under much weaker assumptions on
the measure $\mu$: it suffices that it has a second moment, it does not need
to be admissible, and the distance can be more general than a word distance.
This is proved in~\cite{benoist_quint_TCL_hyperbolic}. Another approach
(which works also in acylindrically hyperbolic groups) has recently been
developed by Mathieu and Sisto~\cite{mathieu_sisto}. If one replaces the word
distance by a nicer distance (for which the Gromov boundary and the Busemann
boundary coincide), then the central limit theorem had been proved earlier by
Bj\"orklund in~\cite{bjorklund_CLT}.

The analyticity of $\mu\mapsto \sigma^2(\mu)$ is proved
in~\cite{HMM_renewal_surface} for surface groups, but it is new in this
generality.

\begin{rmk}
\label{rmk:infinite_support} The assumption of finite support can be somewhat
weakened for the analyticity of $\mu\mapsto \ell(\mu)$ and $\mu\mapsto
\sigma^2(\mu)$. Indeed, it can be replaced by an exponential moment
condition, of the form $\sum \abs{\mu(g)} e^{\alpha\abs{g}}<\infty$, for any
fixed $\alpha>0$. Denote by $\boP(\alpha)$ the set of measures satisfying
this condition (it is a Banach space for the corresponding norm), by
$\boP_1(\alpha)$ its elements with $\sum \mu(g)=1$, and by $\boP_1^+(\alpha)$
its elements which, furthermore, are nonnegative and admissible. Then $\ell$
and $\sigma^2$ are analytic on $\boP_1^+(\alpha)$, i.e., they extend to
analytic functions on a neighborhood of this set in $\boP(\alpha)$. The
proofs are exactly the same as for the finite support case, modulo some
details on analytic functions on Banach spaces which are explained in
Paragraph~\ref{subsec:analytic}.

On the other hand, for the entropy, finite support is essential to our
argument. It is likely that it can be weakened to exponential moment of some
high enough order, using the techniques of~\cite{gouezel_infinite_support},
at the price of significant technical complications.
\end{rmk}

The paper is organized as follows: Section~\ref{sec:prelim} recalls classical
results for random walks on hyperbolic groups. Then, the analyticity of the
rate of escape is proved in Section~\ref{sec:escape}, and the analyticity of
the entropy is proved in Section~\ref{sec:entropy}. Finally,
Section~\ref{sec:CLT} is devoted to the central limit theorem.

\section{Preliminaries}
\label{sec:prelim}

In this paragraph, we recall classical results on random walks on hyperbolic
groups that we will need later on. See~\cite{ghys_hyperbolique} for a general
introduction to hyperbolic groups, and~\cite{ledrappier_entropy_hyperbolic,
haissinsky, gouezel_matheus_maucourant_2} for properties of random walks
there.

\subsection{Hyperbolic groups}

\label{subsec:hyperbolic}

Let $\Gamma$ be a finitely generated group, with a finite symmetric
generating set $S$. The word distance $d=d_S$ is given by
\begin{equation*}
  d(x,y) = \inf\{ n\in \N \st \exists s_1,\dotsc,s_n \in S\text{ with } x^{-1}y=s_1\dotsm s_n\}.
\end{equation*}
This is the graph distance on the Cayley graph of $(\Gamma,S)$.

The group $\Gamma$ is \emph{hyperbolic} if all geodesic triangles in its
Cayley graph are thin, i.e., there exists $\delta\geq 0$ such that each side
of such a triangle is included in the $\delta$-neighborhood of the union of
the two other sides. This geometric definition has a metric counterpart, as
follows. For $x,y,z\in \Gamma$, define the Gromov product of $x$ and $y$ with
basepoint $z$ as
\begin{equation*}
  \gromprod{x}{y}{z} = ( d(x,z)+d(y,z)-d(x,y))/2.
\end{equation*}
It is small if $z$ is close to a geodesic from $x$ to $y$. One should think
of $\gromprod{x}{y}{z}$ as the time during which two geodesics from $z$ to
$x$ and from $z$ to $y$ remain close. The group $\Gamma$ is hyperbolic if and
only, for some $\delta\geq 0$, and for all $x_1,x_2,x_3,z\in \Gamma$,
\begin{equation}
\label{eq:def_hyperb}
  \gromprod{x_1}{x_3}{z} \geq \min( \gromprod{x_1}{x_2}{z}, \gromprod{x_2}{x_3}{z})-\delta.
\end{equation}
A hyperbolic group $\Gamma$ is \emph{nonelementary} if it is not finite nor
virtually $\Z$. Equivalently, it is non-amenable. Then it contains a free
group on two generators.

The Gromov boundary of $\Gamma$, denoted by $\partial \Gamma$, is the set of
equivalence classes of geodesic rays, where two such rays are equivalent if
they stay a bounded distance away. It has a canonical topology, for which
$\Gamma \cup \partial \Gamma$ is a compact space. A sequence $x_n\in \Gamma$
converges to some point on the boundary if and only if
$\gromprod{x_n}{x_m}{z}$ tends to infinity when $n,m\to \infty$ for some, or
equivalently all, basepoint $z$.

The group $\Gamma$ acts on itself by left-multiplication. This action extends
to a continuous action on $\Gamma \cup \partial\Gamma$, and in particular on
$\partial \Gamma$.

\medskip

All the notions we have described up to now are invariant under
quasi-isometry. Hence, they do not depend on the choice of the generating set
$S$ or the word metric $d$ (although the precise value of $\delta$ does). We
turn to a notion that depends on finer details of the distance (this is
necessary to compute the escape rate, which really depends on the distance).

To $y\in \Gamma$, we associate the corresponding normalized distance function
$h_y(x) = d(x,y)-d(e,y)$. It is $1$-Lipschitz and satisfies $h_y(e)=0$. A
horofunction $h$ on $\Gamma$ is a pointwise limit of a sequence $h_{y_n}$
where $y_n\to \infty$. It is still $1$-Lipschitz, satisfies $h(e)=0$, and
$h(\Gamma) \subseteq \Z$. The map $x\mapsto h_x$ is an embedding $\Gamma \to
C^0(\Gamma,\R)$. Identifying $\Gamma$ with its image, and taking its closure
in $C^0(\Gamma,\R)$ for the topology of pointwise convergence, we obtain a
compactification $\Gamma \cup \partial_B \Gamma$, where each $\xi \in
\partial_B \Gamma$ corresponds to a horofunction $h_\xi$ obtained as above.
The compact space $\partial_B \Gamma$ is called the \emph{Busemann boundary}
associated to $(\Gamma,d)$. It really depends on $d$, not just its
quasi-isometry class.

If a sequence $y_n\in \Gamma$ converges in the Busemann compactification, it
also converges in the Gromov compactification. Hence, there is a canonical
continuous projection $\pi_B :\partial_B \Gamma \to \partial \Gamma$, which
is onto but not one-to-one in general. Indeed, the space $\partial_B \Gamma$
is made of $\Z$-valued functions, hence it is totally discontinuous, while
$\partial \Gamma$ is totally discontinuous if and only if $\Gamma$ is
virtually free, see~\cite[Theorem 8.1]{kapovich_survey}. The projection
$\pi_B$ is uniformly finite-to-one
by~\cite{coornaert_papadopoulos_horofunctions}.

The left-action of $\Gamma$ on itself extends to a continuous action on
$\Gamma \cup \partial_B\Gamma$, and in particular on $\partial_B\Gamma$. In
terms of horofunctions, it is given by the following formula:
\begin{equation*}
  h_{g\cdot \xi} (x) = h_\xi(g^{-1} x) - h_\xi(g^{-1}),
\end{equation*}
where the term $-h_\xi(g^{-1})$ ensures that the expression on the right
vanishes for $x=e$, as it should.

We define a distance on the Busemann boundary by
\begin{equation}
\label{eq:def_d}
  d(\xi, \xi') = e^{-n}
\end{equation}
where $n$ is the largest number such that the functions $h_\xi$ and
$h_{\xi'}$ coincide on the ball $B(e,n)$. It is compatible with the topology
of $\partial_B \Gamma$.

Consider a point $\xi \in \partial_B \Gamma$, and a subset $A\subseteq
\Gamma$ such that $\pi_B(\xi)$ does not belong to the closure of $A$ in the
Gromov compactification. Then there exists $C$ such that, for all $x\in A$,
\begin{equation}
\label{eq:dh_bounded_diff}
  \abs{d(e,x) - h_\xi(x)} \leq C.
\end{equation}
Indeed, if $y_n$ tends to $\xi$, then for large enough $n$ a geodesic from
$x\in A$ to $y_n$ intersects a fixed large enough ball around $e$. This
implies that $d(e,x)-h_{y_n}(x)$ remains uniformly bounded. Letting $n$ tend
to infinity, we obtain the claim.

\subsection{Random walks on hyperbolic groups}

Let $\mu$ be an admissible probability measure on a nonelementary hyperbolic
group $\Gamma$, endowed with a word distance $d$. Almost surely, trajectories
$Z_n$ of the corresponding random walk converge in the Gromov
compactification $\Gamma \cup \partial \Gamma$, towards a random point
$Z_\infty \in \partial\Gamma$~\cite[Theorem 7.6]{kaimanovich_poisson}. The
distribution $\nu$ of $Z_\infty$ is a measure on $\partial\Gamma$, called the
exit measure, or hitting measure, or harmonic measure.

An important property of $\nu$ is that it is $\mu$-stationary, i.e., $\mu
*\nu=\nu$, i.e.,
\begin{equation}
\label{eq:def_stationary}
  \nu=\sum_{g\in \Gamma} \mu(g) g_*\nu.
\end{equation}
Indeed, consider the first jump $g$ of the random walk, distributed according
to $\mu$, and then the subsequent trajectories starting from $g$. By
homogeneity, they are distributed like $g Z_{n-1}$, hence their exit
distribution is $g_* \nu$. Averaging with respect to $g$, one
obtains~\eqref{eq:def_stationary}.

Using the contraction properties of the $\Gamma$-action on $\partial \Gamma$,
one checks that $\nu$ is the unique $\mu$-stationary measure on
$\partial\Gamma$~\cite[Theorem 7.6]{kaimanovich_poisson}. Moreover, it is
atomless and has full support.

\medskip

On the other hand, the random walk does not always converge in the Busemann
compactification $\Gamma \cup \partial_B \Gamma$. For instance, if $\Gamma=
F_d \times \Z/2\Z$ (where $F_d$ is a free group), then the Busemann boundary
is made of two copies of $\partial F_d$, corresponding to the two elements of
$\Z/2\Z$. As the random walk keeps jumping between the two sheets $F_d\times
\{0\}$ and $F_d\times \{1\}$, it does not converge.

Nevertheless, by compactness of $\partial_B\Gamma$, there are stationary
measures on $\partial_B \Gamma$. Such measures project under $\pi_B$ to the
unique stationary measure $\nu$ on $\partial \Gamma$. Since the projection
$\pi_B$ is uniformly finite-to-one, there are finitely many ergodic
stationary measures on $\partial_B\Gamma$, their number being bounded by the
cardinality of the fibers. Indeed, the disintegrations above $\nu$ of such
ergodic measures give rise to mutually singular measures on the fibers, and
there are at most $N$ mutually singular measures on a space of cardinality
$N$. Note however that there is no known example of a hyperbolic group with
an admissible measure for which there are more than one stationary measure on
$\partial_B\Gamma$.

One interest of these stationary measures is that they make it possible to
express the escape rate of the random walk. From this point on, we assume
that $\mu$ has a moment of order $1$. Define the Busemann cocycle
\begin{equation}
\label{eq:def_cB}
  c_B : \Gamma \times \partial_B \Gamma \to \R, \quad c_B(g,\xi) = h_\xi(g^{-1}).
\end{equation}
It is an (action) cocycle, i.e., it satisfies the following equation:
\begin{equation}
\label{eq:cocycle_cB}
  c_B(g_1g_2, \xi) = c_B(g_1, g_2\xi) + c_B(g_2, \xi).
\end{equation}
Moreover, if $\nu_B$ is any stationary probability measure on $\partial_B
\Gamma$, one has
\begin{equation}
\label{eq:ell_integral}
  \ell(\mu) = \int_{\Gamma \times \partial_B \Gamma} c_B(g,\xi) \dd\mu(g) \dd\nu_B(\xi),
\end{equation}
see for instance~\cite[Proposition 2.2]{gouezel_matheus_maucourant_2}.

\medskip
Assume now that $\mu$ has finite support. The Green function associated to
$\mu$ is defined by
\begin{equation}
\label{eq:defGmu}
  G_\mu(x,y) = \sum_n \mu^{*n}(x^{-1}y).
\end{equation}
It is the average time that the walk started from $x$ spends at $y$. In
general, the convergence of this series has a simple probabilistic
interpretation, the transience of the random walk. The Green function is also
equal to $\sum_n (Q_\mu^n \delta_y)(x)$, where $Q_\mu$ is the Markov operator
of the random walk, given by $Q_\mu f(x) = \sum \mu(g) f(xg)$. Since $\Gamma$
is non-amenable, the spectral radius of this operator acting on
$\ell^2(\Gamma)$ is $<1$. Hence, the series~\eqref{eq:defGmu} is well defined
as its general term tends to $0$ exponentially fast. Moreover, if $\mu'$ is
any measure close enough to $\mu$, then this series still makes sense, even
if $\mu'$ is not a probability measure any more.

One interest of the Green function is that it is harmonic away from the
diagonal: if $x\neq y$, then $G_\mu(x,y) = \sum_g \mu(g) G_\mu(xg,y)$. This
follows by considering the first jump of the random walk, from $x$ to $xg$,
and then the trajectories from $xg$ to $y$.

\medskip

We will need some results on the Martin boundary of random walks on groups,
as explained for instance in~\cite{sawyer_martin}. One says that a sequence
$y_n \in \Gamma$ converges in the Martin boundary of $(\Gamma,\mu)$ if
$G_\mu(x, y_n)/G(e,y_n)$ converges to a limit for all $x\in \Gamma$. This
defines a compactification of $\Gamma$ by its Martin boundary $\partial_M
\Gamma$, a compact topological space. By definition, when $y_n$ converges to
a point $\xi$ in the Martin boundary, then $G_\mu(x, y_n)/G(e,y_n)$ converges
to a limit denoted by $K_\mu(x,\xi)$, the Martin kernel associated to $\xi$.
Moreover, the Martin kernel depends continuously on $\xi\in
\partial_M\Gamma$. A remarkable property is that
almost every trajectory of the random walk converges in the Martin boundary.
The distribution of the limit is a stationary measure $\nu_M$ on $\partial_M
\Gamma$. Moreover, the entropy satisfies the following formula:
\begin{equation}
\label{eq:hmu}
  h(\mu) = -\int_{\Gamma\times \partial_M\Gamma} \log K_\mu(x^{-1},\xi) \dd\mu(x) \dd\nu_M(\xi).
\end{equation}
This follows from the fact that $(\partial_M \Gamma, \nu_M)$ is a realization
of the Poisson boundary of the random walk, from the general formula
$K_\mu(x,\xi) = \frac{\dd (x_* \nu_M)}{\dd \nu_M}(\xi)$ on the Martin
boundary, and from the integral formula for the entropy on the Poisson
boundary given in~\cite[Theorem~3.1]{kaim_vershik}. Since we will not need
these notions in the remainder of the paper, readers may take the
formula~\eqref{eq:hmu} as a black box.

In general, it is very difficult to describe explicitly the Martin boundary,
i.e., to understand precisely for which sequences $y_n$ one has for all $x$
the convergence of $G_\mu(x, y_n)/G(e,y_n)$. In the specific case we are
considering, i.e., an admissible finitely supported measure on a hyperbolic
group, Ancona has proved in~\cite{ancona2} that the Martin boundary is
exactly the Gromov boundary (and they coincide as topological spaces). In
other words, when $y_n \in \Gamma$ tends to $\xi \in
\partial \Gamma$, then $G_\mu(x, y_n)/G(e,y_n)$ converges to a limit
$K_\mu(x,\xi)$, which is harmonic in $x$ and depends continuously on $\xi$.
Thus, the formula~\eqref{eq:hmu} reads
\begin{equation*}
  h(\mu) = -\int_{\Gamma\times \partial\Gamma} \log K_\mu(x^{-1},\xi) \dd\mu(x) \dd\nu(\xi),
\end{equation*}
where $\nu$ is the unique stationary measure on $\partial\Gamma$ (it
coincides with $\nu_M$ by uniqueness).

We will rather use this formula on the Busemann boundary. The Martin kernel
lifts to $\Gamma\times \partial_B \Gamma$ by the formula
$K_\mu^B(x,\xi)=K_\mu(x,\pi_B(\xi))$. The projection $(\pi_B)_* \nu_B$ of any
stationary measure $\nu_B$ on $\partial_B \Gamma$ is equal to $\nu$. Hence,
the previous formula yields
\begin{equation}
\label{eq:h_integral}
  h(\mu) = -\int_{\Gamma\times\partial_B\Gamma} \log K_\mu^B(x^{-1},\xi) \dd\mu(x) \dd\nu_B(\xi).
\end{equation}

\subsection{Analytic maps between Banach spaces}
\label{subsec:analytic}

To prove the analyticity of $\ell$ and $h$, we need to use analytic mappings
between Banach spaces. Although this is standard, we recall some details for
the convenience of the reader. An excellent reference
is~\cite{mujica_infinite_dim_complex_analysis}.

Let $E$, $F$ be real Banach spaces. A mapping $f$ from an open subset $U$ of
$E$ to $F$ is real-analytic (or simply analytic) if, for every $a\in U$, $f$
has a Taylor expansion around $a$, of the form $f(x) = \sum_m P^m f(a)
(x-a)$, where $P^m f(a)$ is a homogeneous polynomial of degree $m$, i.e., a
map of the form $y \mapsto A(y,\dotsc, y)$ where $A$ is a continuous
$m$-linear map from $E^m$ to $F$. We require that the above series converges
uniformly on some ball around $a$, i.e., $\sum \norm{P^m f(a)} r^m<\infty$
for some $r>0$.

When the same property is satisfied in complex Banach spaces, we say that $f$
is complex-analytic, or analytic, or holomorphic.

Let $f:U\subseteq E\to F$ be a real-analytic map between real Banach spaces.
Using its Taylor series, one obtains an extension $f_\C$ of $f$ where
$f_\C:U_\C\subseteq E_\C \to F_\C$ is holomorphic on a domain $U_\C$ of the
complexification $E_\C = E\otimes \C$, with $U\subseteq U_\C$ . One can
choose to work either with $f$ or $f_\C$, i.e., in real or complex Banach
spaces. However, holomorphic mappings enjoy several remarkable properties
that are not satisfied in the real case, hence we will mainly use the complex
point of view. Notably:
\begin{itemize}
\item A mapping $f$ is holomorphic if and only if it is everywhere
    differentiable over $\C$ (i.e., it is $\R$-differentiable and its
    differential is $\C$-linear)~\cite[Theorem
    13.16]{mujica_infinite_dim_complex_analysis}. For instance, this
    implies without any computation that a composition of holomorphic maps
    is still holomorphic.
\item A mapping $f:U_\C \subseteq E_\C \to F_\C$ is holomorphic if and only
    if it is continuous and, for every $a,b\in E_\C$, the map $z \mapsto
    f(a+bz)$ from an open subset of $\C$ to $F_\C$ is holomorphic where
    defined~\cite[Theorem 8.7]{mujica_infinite_dim_complex_analysis}.
\item A mapping $f:U_\C \subseteq E_\C \to F_\C$ is holomorphic if and only
    if it is continuous and, for every $a,b\in E_\C$ and every $\psi\in
    F_\C'$, the map $z \mapsto \psi \circ f(a+bz)$ from an open subset of
    $\C$ to $\C$ is holomorphic where defined~\cite[Theorem
    8.12]{mujica_infinite_dim_complex_analysis}.
\item A locally uniform pointwise limit of holomorphic mappings is still
    holomorphic, by~\cite[Proposition
    9.13]{mujica_infinite_dim_complex_analysis}.
\end{itemize}

\section{Analyticity of the escape rate}
\label{sec:escape}

Let $\Gamma$ be a nonelementary $\delta$-hyperbolic group, endowed with a
word distance. Increasing $\delta$ if necessary, we may assume that it is an
integer.

In this section, we prove that the map $\mu\mapsto \ell(\mu)$ is analytic on
$\boP_1^+(F)$, as stated in Theorem~\ref{thm:main}. We will even prove the
stronger statement given in Remark~\ref{rmk:infinite_support}. Let
$\alpha>0$. Denote by $\boP(\alpha)$ the complex Banach space of measures
$\mu$ such that $\sum \abs{\mu(g)}e^{\alpha\abs{g}}<\infty$, with the
corresponding norm $\norm{\mu}_\alpha$. We will show that $\mu\mapsto
\ell(\mu)$ is analytic on a neighborhood of the set
$\boP_1^+(\alpha)\subseteq \boP(\alpha)$ of admissible probability measures.

The idea is to use~\eqref{eq:ell_integral}. Thus, we need to exhibit a family
of stationary probability measures on the Busemann boundary $\partial_B
\Gamma$, depending analytically on $\mu$.

We recall that we have defined a distance $d$ on $\partial_B\Gamma$
in~\eqref{eq:def_d}. We write $C^\beta$ for the space of $\beta$-H\"older
continuous function on $(\partial_B \Gamma,d)$, with the norm
\begin{equation*}
  \norm{u}_{C^\beta} = \norm{u}_{C^0} + \sup_{\xi\neq \xi'} \abs{u(\xi)-u(\xi')}/d(\xi,\xi')^\beta.
\end{equation*}.

The main result of this section is the following.

\begin{thm}
\label{thm:stationary_analytic} Let $\alpha>0$ and $\beta>0$. Let $\mu_0 \in
\boP_1^+(\alpha)$ be an admissible probability measure with a finite moment
of order $\alpha$. Then there exist a neighborhood $U$ of $\mu_0$ in
$\boP(\alpha)$ and an analytic map $\Phi : U \to (C^\beta)^*$ such that, for
$\mu \in U\cap \boP_1^+(\alpha)$, the linear form $\Phi(\mu)$ is given by the
integration against a $\mu$-stationary measure on $\partial_B \Gamma$.
\end{thm}

Before proving this theorem, let us see how it implies the analyticity of the
escape rate.

\begin{cor}
The map $u\mapsto \ell(\mu)$, associating to $\mu \in \boP_1^+(\alpha)$ its
escape rate, extends to an analytic map on a neighborhood in $\boP(\alpha)$
of any $\mu_0 \in\boP_1^+(\alpha)$.
\end{cor}
\begin{proof}
Take $\beta<\alpha$. Let $\Phi(\mu)$ be the functional constructed in
Theorem~\ref{thm:stationary_analytic}, on a neighborhood $U\subseteq
\boP(\alpha)$ of $\mu_0\in \boP_1^+(\alpha)$. For $\mu \in U \cap
\boP_1^+(\alpha)$, the linear form $\Phi(\mu)$ corresponds to integration
against a $\mu$-stationary measure. Hence,~\eqref{eq:ell_integral} shows that
the escape rate $\ell(\mu)$ is given by
\begin{equation}
\label{eq:ell_mu_expression}
  \ell(\mu) = \sum_{g\in \Gamma} \mu(g) \Phi(\mu)(\xi \mapsto c_B(g,\xi)).
\end{equation}

We claim that the expression on the right defines an analytic function on
$U$. Let us first show that the sum converges absolutely. We need to estimate
the $C^\beta$ norm of $u_g : \xi \mapsto c_B(g,\xi)$. As $c_B(\cdot, \xi)$ is
$1$-Lipschitz and $c_B(e,\xi)=0$, we get $\abs{c_B(g,\xi)}\leq \abs{g}$,
i.e., $u_g$ is bounded in sup norm by $\abs{g}$. If $d(\xi, \xi') \leq
e^{-\abs{g}}$, we have $u_g(\xi)-u_g(\xi')=0$. On the other hand, if $d(\xi,
\xi')> e^{-\abs{g}}$, we use the sup norm bound by $\abs{g}$ to write
\begin{equation*}
  \abs{u_g(\xi)-u_g(\xi')}/ d(\xi, \xi')^\beta
  \leq 2 \abs{g} / e^{-\beta \abs{g}}
  \leq C e^{\alpha \abs{g}},
\end{equation*}
where the last inequality follows as $\beta<\alpha$. Hence,
$\norm{u_g}_{C^\beta} \leq  C e^{\alpha \abs{g}}$. It follows that
\begin{equation*}
  \sum_{g\in \Gamma} \abs{\mu(g) \Phi(\mu)(u_g)} \leq \sum_{g\in \Gamma} \abs{\mu(g)} C e^{\alpha \abs{g}}
  \leq C \norm{\mu}_\alpha.
\end{equation*}
This shows that the right hand side of~\eqref{eq:ell_mu_expression} is well
defined. Moreover, each such sum, when restricted to a finite subset of
$\Gamma$, is analytic. Since analytic functions are closed under uniform
convergence, as we recalled in Paragraph~\ref{subsec:analytic}, it follows
that the sum over $\Gamma$ is also analytic.
\end{proof}

The rest of this section is devoted to the proof of
Theorem~\ref{thm:stationary_analytic} .

To $\mu \in \boP(\alpha)$ is associated a convolution operator on
$\partial_B\Gamma$, and consequently an operator acting on continuous
functions. It is given by
\begin{equation*}
  L_\mu u(\xi) = \sum_g \mu(g) u(g\cdot \xi).
\end{equation*}
The stationary measure $\Phi(\mu)$ of Theorem~\ref{thm:stationary_analytic}
will be constructed as an eigenfunction for the dual operator $L_\mu^*$
acting on $(C^\beta)^*$, for the eigenvalue $1$. Eigenfunctions corresponding
to isolated eigenvalues which are simple depend analytically on the operator,
hence we should understand the spectral properties of $L_\mu^*$ or,
equivalently, of $L_\mu$.

The usual strategy, dating back to Le Page~\cite{lepage_trick}, is to prove
the following contraction estimate:
\begin{equation}
\label{eq:GLP}
  \sup_{\xi,\xi' \in \partial_B\Gamma} \int \paren*{\frac{d(g\xi, g\xi')}{d(\xi,\xi')}}^\beta \dd\mu^{*n}(g) < 1,
\end{equation}
for some $n>0$ and some $\beta>0$. Such an estimate implies that $L_\mu$ has
a simple eigenvalue at $1$ and no other eigenvalue of modulus $\geq 1$, see
for instance~\cite{bjorklund_CLT} or~\cite{benoist_quint_livre}.
Unfortunately, such an inequality can not hold in general in our context. For
instance, in $\Gamma = F_d \times \Z/2\Z$, consider two horofunctions $\xi$
and $\xi'$ coming from the two sheets $F_d\times \{0\}$ and $F_d \times
\{1\}$. Then $g\xi$ and $g\xi'$ are at distance $1$ for any $g\in \Gamma$,
since they differ on the element $e \times 1$, at distance $1$ of the origin.
Hence,~\eqref{eq:GLP} is always equal to $1$. In this example, one can
nevertheless hope that $L_\mu$ has contraction properties due to another
mechanism, maybe matching for instance $g\xi$ with $g'\xi'$ for another group
element $g'$.

In general, we will not obtain the contraction from an explicit contraction
estimate such as~\eqref{eq:GLP}, but rather from a less explicit spectral
argument.

\begin{prop}
\label{prop:action_continue} For $\mu\in \boP(\alpha)$, the operator $L_\mu$
acts continuously on $C^0$, and on $C^\beta$ when $\beta\leq \alpha$. Its
operator norm is bounded by $\norm{\mu}_\alpha$.
\end{prop}
\begin{proof}
If $\mu$ has finite support, then $L_\mu u$ is the sum of the continuous
functions $\xi \mapsto \mu(g) u(g\cdot \xi)$, bounded by $\abs{\mu(g)}$.
Hence, $L_\mu$ acts continuously on $C^0$ with norm at most $\sum
\abs{\mu(g)} \leq \norm{\mu}_\alpha$. The general case follows by density.

Suppose now that two horofunctions $h_\xi$ and $h_{\xi'}$ coincide on the
ball $B(e,N)$. Then the horofunctions $h_{g\xi}$ and $h_{g\xi'}$ coincide at
least on the ball $B(e, N-\abs{g})$. Therefore,
\begin{equation}
\label{eq:stupid_contraction}
  d(g\xi, g\xi') \leq e^{\abs{g}} d(\xi,\xi').
\end{equation}

Let $u\in C^\beta$. We have
\begin{align*}
  \abs{L_\mu u(\xi) -L_\mu u(\xi')}
  &\leq \sum \abs{\mu(g)} \abs{u(g\xi)-u(g\xi')}
  \leq \sum \abs{\mu(g)} \norm{u}_{C^\beta} d(g \xi, g \xi')^\beta
  \\&
  \leq \sum \abs{\mu(g)} \norm{u}_{C^\beta} e^{\beta \abs{g}}d(\xi, \xi')^\beta
  = \norm{\mu}_\beta \norm{u}_{C^\beta} d(\xi, \xi')^\beta
  \\&
  \leq \norm{\mu}_\alpha \norm{u}_{C^\beta} d(\xi, \xi')^\beta.
\end{align*}
This shows that $L_\mu u$ is again $\beta$-H\"older continuous, with H\"older
constant at most $\norm{\mu}_\alpha \norm{u}_{C^\beta}$.
\end{proof}

This computation does not give any contraction. To get contraction, we should
show that $d(g\xi, g\xi')$ is smaller than $d(\xi,\xi')$ for most $g$.
Starting from the fact that two horofunctions $h_\xi$ and $h_{\xi'}$ coincide
on a ball $B(e,N)$, we thus need to show that they coincide on a bigger ball
centered at $g^{-1}$. This is essentially the content of the following lemma.

\begin{lem}
\label{lem:coincide} Let $h_1$ and $h_2$ be two horofunctions, which coincide
on a ball $B(e,N)$. Let $x\in B(e,N-10\delta)$. Then $h_1$ and $h_2$ coincide
on the ball $B(x, N+h_1(x)-10\delta)$.
\end{lem}
This lemma is trivial for $x=e$, or more generally when $h_1(x)=-d(e,x)$: in
this case, the ball $B(x, N+h_1(x)-10\delta)$ is included in $B(e,N)$, where
we already know that $h_1$ and $h_2$ coincide. On the other hand, it gives
new and useful information for instance when $h_1(x)$ is positive. This lemma
follows easily from the arguments
in~\cite{coornaert_papadopoulos_horofunctions}, as we explain now. The idea
is that geodesics from a point $y\in B(x, N+h_1(x)-10\delta)$ to the point at
infinity directed by $h_1$ have to enter $B(e,N)$, hence the value of $h_1$
on $y$ is determined by its values on $B(e,N)$. The reader should keep this
geometric picture in mind, although the rigorous formalization in terms of
Gromov products is rather tedious.

\begin{proof}
We will use the following easy inequality on Gromov products. For any $u$,
$v$ and $w$, one has $d(w,v)\geq d(w,u)-d(u,v)$. Hence
\begin{equation*}
  \gromprod{u}{v}{w} = ( d(w,u)+d(w,v)-d(u,v))/2 \geq d(w,u)-d(u,v).
\end{equation*}

Take $x\in B(e,N-10\delta)$. As $h_1$ is a horofunction, it is a pointwise
limit of normalized distance functions. Hence, we may take $z$ with
$d(e,z)>2N$ such that $h_1(y) = h_z(y)$ for all $y\in B(e,2N)$, where $h_z(y)
= d(z,y) - d(z,e)$. Write $d(e,z)=M+N$ with $M>N$. Then
\begin{equation*}
  \gromprod{e}{x}{z}
  \geq d(z,e)-d(e,x)
  \geq M+N-(N-10\delta) = M+10 \delta.
\end{equation*}
Consider $y\in B(x, N+h_1(x)-10\delta)$. As $d(z,
x)=d(z,e)+h_z(x)=M+N+h_1(x)$, we obtain
\begin{equation*}
  \gromprod{x}{y}{z} \geq d(z,x)- d(y,x) \geq M+N+h_1(x) - (N+h_1(x)-10\delta)
  = M+10\delta.
\end{equation*}
The inequality~\eqref{eq:def_hyperb} characterizing hyperbolic spaces entails
\begin{equation}
\label{eq:eyx1}
  \gromprod{e}{y}{z} \geq M+9\delta.
\end{equation}

Consider $a$ and $b$ the points on geodesics from $z$ to $e$ and from $z$ to
$y$, at distance $M+9\delta$ of $z$. They belong to $\Gamma$ since $\delta\in
\N$ by assumption. Moreover, $d(e,a)=N-9\delta$. We have $\gromprod{e}{a}{z}
= M+9\delta$ and $\gromprod{y}{b}{z}=M+9\delta$. Combining these two
inequalities with hyperbolicity, and using~\eqref{eq:eyx1}, we get
\begin{equation*}
  \gromprod{a}{b}{z} \geq M+7\delta.
\end{equation*}
Hence,
\begin{equation*}
  d(a,b) = d(a,z)+d(b,z) -2\gromprod{a}{b}{z} \leq 2(M+9\delta)-2(M+7\delta)=4\delta.
\end{equation*}
In particular, $d(b,e) \leq d(a,e)+4\delta \leq N-5\delta$. This implies that
$b\in B(e,N)$, so that $h_1(b)=h_2(b)$.

As $b$ is on a geodesic segment from $z$ to $y$, we have
\begin{equation*}
  h_1(y)=h_z(y)=h_z(b)+d(b,y)=h_1(b)+d(b,y).
\end{equation*}
Moreover, as $h_2$ is $1$-Lipschitz, $h_2(y) \leq h_2(b)+d(b,y)$. This shows
that $h_2(y)\leq h_1(y)$.

We have proved that $h_2 \leq h_1$ on $B(x, N+h_1(x) -10\delta)$. Since
everything is symmetric, the reverse inequality and the equality follow.
\end{proof}

A weak form of contraction of the convolution operator (called a
Doeblin-Fortet inequality) follows using the standard trick of Le
Page~\cite{lepage_trick}, as explained for instance
in~\cite[Lemma~12.6]{benoist_quint_livre}.

\begin{lem}
\label{lem:LY} Let $\mu\in \boP_1^+(\alpha)$. There exist $n>0$, $\beta\leq
\alpha/2$, $C>0$ and $\rho<1$ such that, for any $u\in C^\beta$,
\begin{equation}
\label{eq:LY}
  \norm{L_\mu^n u}_{C^\beta} \leq \rho \norm{u}_{C^\beta} + C \norm{u}_{C^0}.
\end{equation}
\end{lem}
\begin{proof}
Let $\beta\leq \alpha/2$. Fix $n\in \N$, and $N>0$. As the sup norm control
$\norm{L\mu^n u}_{C^0} \leq \norm{u}_{C^0}$ is trivial, we just have to
estimate the $\beta$-H\"older constant of $L_\mu^n u$. Given $\xi$ and
$\xi'$, we should bound $\abs{L_\mu^n u(\xi) -L_\mu^n u(\xi')}$. If
$d(\xi,\xi')>e^{-N}$, we simply write
\begin{equation*}
  \abs{L_\mu^n u(\xi) -L_\mu^n u(\xi')} \leq 2\norm{L_\mu^n u}_{C^0}
  \leq 2 \norm{u}_{C^0} \leq d(\xi,\xi')^\beta \cdot 2\cdot e^{\beta N}  \norm{u}_{C^0}.
\end{equation*}
This is compatible with the inequality we seek, for a large $C \geq 2e^{\beta
N}$.

Assume now that $d(\xi,\xi') =e^{-M} \leq e^{-N}$, i.e., the two
horofunctions $h_\xi$ and $h_{\xi'}$ coincide on the ball $B(e,M)$ with
$M\geq N$. When $\abs{g} \leq N-10\delta$, Lemma~\ref{lem:coincide} implies
that $h_\xi$ and $h_{\xi'}$ coincide on the ball $B(g^{-1},
M+h_\xi(g^{-1})-10\delta)$. We recall that $g\cdot \xi$ satisfies
$h_{g\xi}(x)= h_\xi(g^{-1}x)-h_\xi(g^{-1})$. Hence, $h_{g\xi}$ and
$h_{g\xi'}$ coincide on the ball $B(e, M+h_\xi(g^{-1})-10\delta)$. This
yields $d(g\xi, g\xi') \leq e^{10\delta -h_\xi(g^{-1})} d(\xi, \xi')$. If
$\abs{g} > N-10\delta$, we simply use the trivial inequality $d(g\xi, g\xi')
\leq e^{\abs{g}} d(\xi, \xi')$ proved in~\eqref{eq:stupid_contraction}
instead.

We obtain
\begin{multline*}
  \abs{L_\mu^n u(\xi) - L_\mu^n u(\xi')}
  \leq \sum_g \mu^{*n}(g) \abs{u(g\xi)-u(g\xi')}
  \leq \sum_g \mu^{*n}(g) \norm{u}_{C^\beta} d(g\xi, g\xi')^\beta
  \\
  \leq \norm{u}_{C^\beta} d(\xi, \xi')^\beta
    \paren*{ \sum_{\abs{g} \leq N-10\delta}\mu^{*n}(g) e^{\beta(10\delta -h_\xi(g^{-1}))}
    +\sum_{\abs{g}> N-10\delta} \mu^{*n}(g)e^{\beta \abs{g}}}.
\end{multline*}
If the term between parenthesis on the last line is bounded by $\rho<1$,
uniformly in $\xi \in \partial_B \Gamma$, then we get $\abs{L_\mu^n u(\xi) -
L_\mu^n u(\xi')} \leq \rho \norm{u}_{C^\beta} d(\xi, \xi')^\beta$, which is
the desired bound for the pair $\xi,\xi'$.

Hence, to conclude, it suffices to obtain such a bound by $\rho <1$ as above.
It is even sufficient to obtain it for $N=\infty$, since the same bound (with
a slightly larger $\rho'\in (\rho,1)$) then follows for any large enough
finite $N$. Finally, it suffices to find $\beta\leq\alpha/2$ and $n$ such
that, uniformly in $\xi \in \partial_B \Gamma$,
\begin{equation}
\label{eq:eqexp}
  \sum_g \mu^{*n}(g) e^{\beta(10\delta -h_\xi(g^{-1}))} \leq \rho<1.
\end{equation}

We use the inequality $e^t \leq 1+t+t^2 e^{\abs{t}}$. Moreover, there exists
$C$ such that $t^2 \leq C e^{\alpha \abs{t}/2}$ for all real $t$. Hence, the
sum in the above equation is bounded by
\begin{multline*}
  \sum_g \mu^{*n}(g) \paren*{ 1+ \beta (10\delta -h_\xi(g^{-1})) +
  \beta^2 (10\delta -h_\xi(g^{-1}))^2 e^{\beta (10 \delta+\abs{g})}}
  \\
  \leq 1- \beta \paren*{\sum_g  \mu^{*n}(g)(h_\xi(g^{-1})-10\delta)} +
    \beta^2 C \paren*{\sum_g  \mu^{*n}(g)e^{(\beta+\alpha/2)(10 \delta+\abs{g})}}.
\end{multline*}
When $\beta \leq \alpha/2$, the last sum is finite and bounded by $e^{10
\alpha \delta} \norm{\mu^{*n}}_\alpha$. Hence, when $\beta$ tends to $0$, the
last term is $O(\beta^2)$, and is negligible with respect to the first one,
of order $\beta$. If the first term is strictly negative,~\eqref{eq:eqexp}
follows for small enough $\beta$. Therefore, it suffices to show that
\begin{equation}
\label{eq:lineaire}
  \sum_g  \mu^{*n}(g)(h_\xi(g^{-1})-10\delta) \geq K>0,
\end{equation}
uniformly in $\xi \in \partial_B \Gamma$. The term on the left of this
equation looks closely like the escape rate of the random walk, which is
strictly positive. The difficulty is that we want a pointwise uniform
inequality, not an averaged version.

This kind of situation has already been encountered several times in the
literature, hence efficient tools are available. We recall for instance
Theorem 2.9 in~\cite{benoist_quint_livre}. Let $\Gamma$ be a countable group
acting continuously on a compact space $X$, let $\mu$ be a probability
measure on $\Gamma$, and let $c:\Gamma \times X \to \R$ be a continuous
cocycle, i.e., a continuous function satisfying $c(g_1 g_2, x) = c(g_1, g_2
x)+ c(g_2, x)$. Assume that $\sum_{g\in \Gamma} \mu(g) \sup_{x\in X}
\abs{c(g,x)} < \infty$. Assume also that there exists $\ell\in \R$ such that,
for any $\mu$-stationary probability measure $\nu$ on $X$, $\int_{\Gamma
\times X} c(g,x) \dd\mu(g) \dd\nu(x)= \ell$. Then, as $n \to \infty$,
\begin{equation*}
  \frac{1}{n}\sum_{g\in \Gamma}\mu^{*n}(g) c(g,x) \to \ell,
\end{equation*}
uniformly in $x\in X$.

We will apply this statement to the Busemann cocycle on $X=\partial_B
\Gamma$, given by $c_B(g, \xi)=h_\xi(g^{-1})$. In this case, for any
stationary measure $\nu$ on $\partial_B \Gamma$,~\eqref{eq:ell_integral}
states that
\begin{equation*}
  \int_X c_B(g,\xi)\dd\mu(g)\dd\nu(\xi) = \ell.
\end{equation*}
Hence, we deduce from~\cite[Theorem 2.9]{benoist_quint_livre} that
\begin{equation*}
  \frac{1}{n}\sum_{g\in\Gamma} \mu^{*n}(g) c_B(g,\xi) \to \ell,
  \text{ uniformly in }\xi\in \partial_B\Gamma.
\end{equation*}
As $\mu$ is admissible and $\Gamma$ is non-amenable, $\ell>0$. Hence, if $n$
is large enough, we have for all $\xi \in \partial_B \Gamma$ the inequality
\begin{equation*}
  \sum_g \mu^{*n}(g) h_\xi(g^{-1}) \geq n \ell/2.
\end{equation*}
This implies~\eqref{eq:lineaire} if $n$ is large enough so that $n\ell/2 >
10\delta$.
\end{proof}

The group $\Gamma$ acts continuously on the compact set $\partial_B \Gamma$.
By Zorn's lemma, there exists a nonempty compact invariant subset of
$\partial_B \Gamma$ which is minimal among such sets. It is simply called a
minimal subset of $\partial_B \Gamma$. An essential feature of such a set
$\partial_B' \Gamma$ is that, for any $x \in \partial_B' \Gamma$, then
$\Gamma x$ is dense in $\partial_B' \Gamma$. Indeed, otherwise, the closure
of $\Gamma x$ would be a nonempty compact invariant strict subset of
$\partial_B' \Gamma$, contradicting its minimality.

\begin{prop}
\label{prop:spectral_description} Let $\mu\in \boP_1^+(\alpha)$. Consider a
nonempty compact subset $\partial'_B \Gamma \subseteq \partial_B \Gamma$
which is minimal for the $\Gamma$-action. Then, for small enough $\beta$, the
operator $L_\mu$ acting on $C^\beta(\partial'_B \Gamma)$ has a simple
eigenvalue at $1$, finitely many eigenvalues of modulus $1$ (they are simple
and do not have nontrivial Jordan blocks) and the rest of its spectrum is
contained in a disk $D(0,r)$ for some $r<1$.
\end{prop}
\begin{proof}
In the proof, we let $L_\mu$ act on $C^\beta(\partial'_B \Gamma)$. On this
space, the estimates of Lemma~\ref{lem:LY} readily follow from the
corresponding ones on $C^\beta(\partial_B \Gamma)$. Let $n$ and $\beta$ be
given by Lemma~\ref{lem:LY}. In the inequality~\eqref{eq:LY}, the term $\rho
\norm{u}_{C^\beta}$ heuristically corresponds to a part of $L_{\mu}$ with
spectral radius at most $\rho^{1/n}$, while the term $C \norm{u}_{C^0}$ would
come from a compact part (as the inclusion of $C^\beta$ in $C^0$ is compact),
which should only add discrete eigenvalues. This intuition is made precise by
a theorem of Hennion~\cite{hennion}: It entails that~\eqref{eq:LY} implies
that the spectrum of $L_{\mu}$ in $\{z\in \C \st \abs{z}>\rho^{1/n}\}$ is
made of isolated eigenvalues of finite multiplicity.

In particular, by discreteness, there are finitely many eigenvalues of
modulus $\geq 1$, and the rest of the spectrum is contained in a disk
$D(0,r)$ for some $r<1$. As the iterates of $L_{\mu}$ on $C^0$ are uniformly
bounded, there is no eigenvalue with modulus $>1$. Moreover, the eigenvalues
of modulus $1$ have no Jordan block.

Let $u$ be a nonzero eigenfunction of $L_{\mu}$ for an eigenvalue $\rho$ of
modulus $1$. Then $v=\abs{u}$ satisfies
\begin{equation}
\label{eq:vabsu}
  v=\abs{u} =\abs{L^k_{\mu} u} \leq L^k_{\mu} \abs{u}=L^k_{\mu} v.
\end{equation}
Consider $\xi \in \partial'_B \Gamma$ such that $v(\xi)$ is maximal. Then the
previous inequality implies that $v(g\xi)=v(\xi)$ for all $g$ in the
semigroup generated by the support of $\mu$. By admissibility, this is true
for all $g\in \Gamma$. The orbit of $\xi$ is dense in $\partial'_B \Gamma$ by
minimality. Thus, $v$ is constant.

The equality in~\eqref{eq:vabsu} also implies that all the complex numbers
$u(g\xi)$ for $g\in \supp(\mu^{*k})$ have the same phase. Hence, $u$ is
constant on $(\supp \mu^{*k})\xi$.

We claim that
\begin{equation}
\label{eq:exists_good_N}
  \text{there exists $N>0$ with $\mu^{*N}(e)>0$.}
\end{equation}
Indeed, fix $x\neq e$. Then, by admissibility, there is $\ell_1>0$ such that
$\mu^{*\ell_1}(x)>0$, and $\ell_2>0$ such that $\mu^{*\ell_2}(x^{-1})>0$.
Then $\mu^{*(\ell_1+\ell_2)}(e)>0$.

Fix such an $N$. Then, for any $j$, the sequence $(\supp \mu^{*(kN+j)})(\xi)$
increases with $k$, towards a limiting set $A_j \subseteq \partial'_B
\Gamma$. The function $u$ is constant on each set $\overline{A_j}$. Moreover,
for $j\in \Z/N\Z$, one has $(\supp \mu) \cdot A_j = A_{j+1}$. As $L_\mu u =
\rho u$, it follows that $\rho u_{\restr\overline{A_j}} =
u_{\restr\overline{A_{j+1}}}$. Finally, one gets $u(x) = \rho^{j} u(\xi)$ for
$x\in \overline{A_j}$. As $\bigcup_{j} \overline{A_j} = \partial'_B \Gamma$
by minimality, this shows that $u$ is determined by its value at $\xi$, and
therefore that the $\rho$-eigenspace is at most $1$-dimensional.

For $\rho=1$, the eigenspace is exactly $1$-dimensional, as it contains the
constant functions.
\end{proof}

\begin{ex*}
On $\Gamma = F_2 \times \Z/2\Z$, consider the uniform measure $\mu$ on
$(e,1)$, $(a,1)$, $(a^{-1},1)$, $(b,1)$ and $(b^{-1},1)$, where $a$ and $b$
are the canonical generators of $F_2$. It is admissible. Let $\nu$ denote its
(unique) stationary measure on the Gromov boundary $\partial\Gamma \simeq
\partial F_2$. The Busemann boundary $\partial_B \Gamma$ is canonically
isomorphic to $\partial F_2 \times \Z/2\Z$, and the $\Gamma$ action on
$\partial_B \Gamma$ is minimal. Write $\nu_0$ and $\nu_1$ for $\nu \otimes
\delta_0$ and $\nu \otimes \delta_1$. As the measure $\mu$ exchanges the
sheets $F_2 \times \{0\}$ and $F_2 \times \{1\}$, the operator $(L_\mu)^*$
maps $\nu_0$ to $\nu_1$ and $\nu_1$ to $\nu_0$. Dually, denoting by $u_i$ the
characteristic function of $\partial F_2 \times \{i\}$ for $i=0,1$, the
operator $L_\mu$ maps $u_0$ to $u_1$ and $u_1$ to $u_0$. The constant
function $u=u_0+u_1=2$ is invariant under $L_\mu$. Moreover, this operator
also has a simple eigenvalue at $-1$, for the eigenfunction $u_0-u_1$. There
is no other eigenvalue of modulus $1$. Dually, the measure $\nu_0+\nu_1$ is
an eigenmeasure of $(L_\mu)^*$ for the eigenvalue $1$ and $\nu_0-\nu_1$ is an
eigenmeasure for the eigenvalue $-1$.
\end{ex*}

\begin{proof}[Proof of Theorem~\ref{thm:stationary_analytic}]
It suffices to prove the theorem for small enough $\beta$, as $(C^\beta)^*
\subseteq (C^{\beta'})^*$ when $\beta \leq \beta'$. We will use basic results
of spectral theory in Banach spaces, as explained for instance
in~\cite{kato_pe}.

Take $\mu_0\in \boP_1^+(\alpha)$. Proposition~\ref{prop:spectral_description}
applies to $\mu_0$. Hence, for small enough $\beta$, one can decompose the
space $C^{\beta}(\partial'_B \Gamma)$ as a direct sum of finite-dimensional
eigenspaces $E_\rho$ associated to eigenvalues $\rho$ with modulus $1$, and
an infinite-dimensional subspace $E_{<1}$ on which $L_{\mu_0}^n$ tends
exponentially fast to $0$. Let $\Pi_{\mu_0}$ be the eigenprojection for the
eigenvalue $1$, i.e., the projection on $E_1$ with kernel $E_{<1}\oplus
\bigoplus_{\rho\neq 1} E_\rho$. Then we claim that, for any $u\in
C^{\beta}(\partial'_B \Gamma)$,
\begin{equation}
\label{eq:limite}
  \frac{1}{k}\sum_{i=0}^{k-1} L_{\mu_0}^i u \to \Pi_{\mu_0} u.
\end{equation}
Indeed, this is clear on each $E_\rho$ and on $E$ separately, and the general
result follows. By~\cite{kato_pe}, the projection $\Pi_{\mu_0}$ is also given
by the formula
\begin{equation}
\label{eq:proj_spectral}
  \Pi_{\mu_0} u = \frac{1}{2\ic \pi} \int_{\boC} (zI-L_{\mu_0})^{-1} \dd z,
\end{equation}
where $\boC$ is any small enough circle around $1$.

The map $\mu \mapsto L_\mu$ is linear and continuous by
Proposition~\ref{prop:action_continue} (and therefore analytic). For $\mu$
close enough to $\mu_0$, the operator $L_\mu$ is close to $L_{\mu_0}$. Hence,
the spectral description given by Proposition~\ref{prop:spectral_description}
for $L_{\mu_0}$ persists for $L_\mu$ by spectral continuity results for
isolated simple eigenvalues, see~\cite{kato_pe}: the operator $L_\mu$ has a
unique eigenvalue close to $1$, it is simple, and the corresponding spectral
projection $\Pi_{\mu}$ is given by the formula~\eqref{eq:proj_spectral} (with
$\mu_0$ replaced by $\mu$). Moreover, $\mu\mapsto \Pi_{\mu}$ is analytic.
When $\mu\in \boP_1^+(\alpha)$, the operator $L_\mu$ is a convolution
operator with a probability measure, hence $L_\mu 1 = 1$. In particular, the
corresponding eigenvalue is $1$.

Let $\xi_0 \in \partial'_B \Gamma$. We define a linear form $\Phi(\mu)$ on
$C^\beta(\partial'_B \Gamma)$, for $\mu$ close to $\mu_0$, by $\Phi(\mu)(u) =
(\Pi_\mu u)(\xi_0)$. Then $\mu \mapsto \Phi(\mu)$ is analytic from $U$ to
$(C^\beta(\partial'_B \Gamma))^*$. As H\"older functions on $\partial_B
\Gamma$ restrict to H\"older functions on $\partial'_B \Gamma$, $\Phi(\mu)$
can be considered as a linear form on $C^\beta(\partial_B \Gamma)$.
Equivalently, we implicitly compose $\Phi(\mu)$ with the continuous inclusion
$(C^\beta(\partial'_B \Gamma))^* \subseteq (C^\beta(\partial_B \Gamma))^*$.
Hence, $\Phi(\mu)\in (C^\beta(\partial_B \Gamma))^*$.

It remains to show that the linear form $\Phi(\mu)$ is the integration
against a stationary probability measure on $\partial'_B \Gamma$ when $\mu\in
\boP_1^+(\alpha)$. In this case, $\Phi(\mu)(u)$ is again given by the
formula~\eqref{eq:limite} (with $\mu_0$ replaced by $\mu$), as the above
discussion also applies to $\mu$. For any nonnegative H\"older function $u$,
one gets $0\leq \Phi(\mu)(u) \leq \norm{u}_{C^0}$. By Stone-Weierstrass
theorem, H\"older functions are dense in $C^0$. We deduce that $\Phi(\mu)$
extends to a positive linear form on continuous functions, i.e., a positive
measure on $\partial'_B \Gamma$. As $\Phi(\mu)1=1$, it is a probability
measure. Finally, as $\Phi(\mu)(L_\mu u) = \Phi(\mu)(u)$, it is stationary.
\end{proof}

\section{Analyticity of the entropy}
\label{sec:entropy}

Let $\Gamma$ be a nonelementary hyperbolic group. We fix once and for all a
finite subset $F$ of $\Gamma$. We denote by $\boP(F)$ the set of functions
$\mu:F\to \C$, and by $\boP_1^+(F)$ its subset made of admissible probability
measures. We also fix a reference measure $\mu_0\in \boP_1^+(F)$. Note that
the support of $\mu_0$ may be a proper subset of $F$. In this section, we
prove the entropy part of Theorem~\ref{thm:main}:

\begin{thm}
\label{thm:entropy_analytic} The map $\mu\mapsto h(\mu)$, associating to $\mu
\in \boP_1^+(F)$ its entropy, extends to an analytic map on a neighborhood of
$\mu_0$ in $\boP(F)$.
\end{thm}

The idea is to start from~\eqref{eq:h_integral}, and to show that the lift
$K_\mu^B(x,\xi)$ of the Martin kernel to the Busemann boundary depends
analytically on $\mu$.

\subsection{Strong Ancona inequalities}

A fundamental tool to study the Green function and the Martin kernel in
hyperbolic groups is an inequality due to Ancona, showing that the Green
function is essentially multiplicative along geodesics. In other words,
typical trajectories of the random walk tend to follow geodesics. We will use
repeatedly a quantitative version of such inequalities, that we describe in
this paragraph. As it should hold uniformly on a neighborhood of $\mu_0$, we
first describe such a convenient neighborhood.

\medskip

As $\Gamma$ is non-amenable, the spectral radius of the convolution operator
by $\mu_0$ on $\ell^2(\Gamma)$ is $<1$. We fix $\epsilon_0>0$ small enough so
that the measure $\bar\mu = \mu_0+\epsilon_0\sum_{g\in F}\delta_g$ also has a
spectral radius $<1$.

We denote by $\bar U \subseteq \boP(F)$ the set of functions supported on $F$
with $\abs{\mu(g)-\mu_0(g)} \leq \epsilon_0$ for all $g\in F$, and by $U$ its
interior. We also write $\bar U^+$ for the nonnegative elements of $\bar U$,
and $\bar U^+_1$ for the admissible probability measures in $\bar U$. As all
functions in $\bar U$ are dominated by $\bar\mu$, they all have a spectral
radius $<1$, uniformly. Moreover, the Green function (defined
in~\eqref{eq:defGmu}) of $\mu\in \bar U^+$ satisfies
\begin{equation*}
  G_{\mu}(x,y) \leq G_{\bar\mu}(x,y)
\end{equation*}
for all $x,y\in \Gamma$.

If $\epsilon_0$ is small enough, all measures $\mu\in \bar U^+$ satisfy
$\mu(g) \geq \mu_0(g)/2$ for all $g\in \supp \mu_0$. As $\mu_0$ is
admissible, it follows that $G_\mu(x,y) \geq e^{-C d(x,y)}$ for all $x,y\in
\Gamma$, uniformly in $\mu$. As trajectories from $x$ to $y$ and trajectories
from $y$ to $z$ can be concatenated to form trajectories from $x$ to $z$, we
deduce the following Harnack inequalities, uniformly in $x,y,z\in \Gamma$ and
$\mu \in\bar U^+$:
\begin{equation}
\label{eq:harnack}
  e^{- C d(x,y)} \leq \frac{G_\mu(x,z)}{G_\mu(y,z)} \leq e^{C d(x,y)},\quad
  e^{- C d(y,z)} \leq \frac{G_\mu(x,y)}{G_\mu(x,z)} \leq e^{C d(y,z)}.
\end{equation}

For $\mu \in \bar U^+$, we say that a function $u:\Gamma \to \R$ is
$\mu$-harmonic in a domain $A$ of $\Gamma$ if, for all $x\in A$,
\begin{equation*}
  u(x) = \sum_g \mu(g) u(xg).
\end{equation*}
For instance, the Green function $x\mapsto G_\mu(x,y)$ is $\mu$-harmonic on
$\Gamma \setminus \{y\}$.

\medskip

Consider a geodesic $\gamma$ in $\Gamma$, with length $D \geq 0$, from
$x_0=\gamma(0)$ to $y_0=\gamma(D)$. We define a domain of points close to the
beginning of $\gamma$ as $I^-(\gamma)=\{x\in \Gamma \st
\gromprod{x_0}{x}{y_0} \geq D-10\delta\}$. In the same way, let
$I^+(\gamma)=\{y\in \Gamma \st \gromprod{y_0}{y}{x_0} \geq D-10\delta\}$.

The quantitative Ancona inequalities we will use are the following:
\begin{prop}
\label{prop:ineq_ancona_strong} There exist $C>0$ and $D_0$ such that, for
all $\mu\in \bar U^+$, the following holds. Let $\gamma$ be a geodesic
segment in $\Gamma$, with length $D \geq D_0$. Let $u$ and $v$ be two
nonnegative functions on $\Gamma$ which satisfy the following outside of
$I^+(\gamma)$: they are strictly positive, $\mu$-harmonic, and bounded by a
finite linear combination of functions $G_{\mu}(\cdot, y)$ where $y\in
I^+(\gamma)$. Then, for all $x, x'\in I^-(\gamma)$
\begin{equation*}
  \abs*{ \frac{u(x)/u(x')}{v(x)/v(x')} -1} \leq Ce^{-C^{-1} D}.
\end{equation*}
\end{prop}
\begin{proof}
This follows from the arguments in~\cite{gouezel_lalley}. Indeed, the
conclusion of Lemma 4.4 there is a consequence of the fact that the spectral
radius is bounded away from $1$ on $\bar U^+$. Then all the arguments up to
the end of Theorem 4.6 apply verbatim, even for general $\mu$-harmonic
functions. Note that, while it is assumed in the statements
of~\cite{gouezel_lalley} that the measure is symmetric, this is only used in
the proof of Lemma 4.4 there (which is trivial in our context, even without
symmetry).
\end{proof}

One important tool in this proof is the \emph{Ancona inequality} (Theorem 4.1
in~\cite{gouezel_lalley}), saying that the Green function is multiplicative
up to a constant along geodesics. We will use the following version:
\begin{lem}
\label{lem:Ancona} For any $K>0$, there exists $C>0$ with the following
property. Let $\mu\in \bar U^+$. Consider three points $x$, $y$, $z$ in
$\Gamma$ such that $z$ is within distance $K$ of a geodesic segment between
$x$ and $y$. Then
\begin{equation*}
  C^{-1} G_\mu(x,z) G_\mu(z,y) \leq G_{\mu}(x,y)\leq C G_\mu(x,z) G_\mu(z,y).
\end{equation*}
\end{lem}
To illustrate the use of Proposition~\ref{prop:ineq_ancona_strong}, let us
explain how, conversely, it implies Ancona inequalities. We stress that this
is \emph{not} the logical order, as the proof of
Proposition~\ref{prop:ineq_ancona_strong} uses the lemma.
\begin{proof}
By Harnack inequalities, it suffices to prove the lemma when $K=0$, i.e., $z$
belongs to a geodesic between $x$ and $y$.

If $z$ is $D_0$-close to $x$ or $y$, then Ancona inequalities are trivial
thanks to Harnack inequalities~\eqref{eq:harnack}. Otherwise, along a
geodesic from $x$ to $y$ containing $z$, consider the point $x_0$ between $x$
and $z$ at distance $D=D_0$ of $z$, and similarly the point $y_0$ between $z$
and $y$ with $d(z,y_0)=D_0$. Let $\gamma$ be the restriction of the geodesic
to $[x_0 y_0]$. Then $x, x_0 \in I^-(\gamma)$ and $y,y_0 \in I^+(\gamma)$.
The functions $u = G_\mu(\cdot,y)$ and $v=G_\mu(\cdot,y_0)$ satisfy the
assumptions of the proposition. We deduce
\begin{equation*}
  \abs*{ \frac{G_\mu(x,y)/G_\mu(x_0, y)}{G_\mu(x, y_0)/G_\mu(x_0,y_0)}-1}\leq Ce^{-C^{-1}D_0}.
\end{equation*}
In particular, the fraction is bounded. As $x_0$ and $y_0$ are within bounded
distance of $z$, this fraction coincides, up to a multiplicative constant,
with
\begin{equation*}
  \frac{G_\mu(x,y)/ G_\mu(z,y)}{G_\mu(x,z) / G_\mu(z,z)}.
\end{equation*}
Its boundedness shows that $G_\mu(x,y) \leq C G_\mu(x,z) G_\mu(z,y)$.
Conversely, the other inequality $G_\mu(x,y) \geq C^{-1}G_\mu(x,z)
G_\mu(z,y)$ is trivial since trajectories from $x$ to $z$ and from $z$ to $y$
can be concatenated to form trajectories from $x$ to $y$.
\end{proof}

\begin{rmk}
\label{rmq:Martin_borne} Proposition~\ref{prop:ineq_ancona_strong} applies to
the Martin kernel $x\mapsto K_\mu(x,\xi)$ when $\xi \notin \overline{\Gamma
\setminus I^+(\gamma)}$. It is $\mu$-harmonic and positive everywhere, what
remains to be checked is that it is bounded by $CG_{\mu}(x,y_0)$. Let $y_n\in
I^+(\gamma)$ tend to $\xi$. Then a geodesic from $x$ to $y_n$ passes within
bounded distance of $y_0$, uniformly in $n$ and in $x$ in $\Gamma \setminus
I^+(\gamma)$. Thus, Ancona inequalities give
\begin{equation*}
  G_{\mu}(x, y_n)\leq C G_{\mu}(x, y_0)G_{\mu}(y_0,y_n).
\end{equation*}
By Harnack inequalities, $G_\mu(y_0,y_n) \leq C G_\mu(e,y_n)$. Hence,
\begin{equation*}
  G_\mu(x,y_n)/G_\mu(e,y_n) \leq C G_{\mu}(x, y_0).
\end{equation*}
Letting $n$ tend to infinity, we obtain $K_\mu(x,\xi) \leq C G_{\mu}(x, y_0)$
as desired.
\end{rmk}

When $A$ is a subset of $\Gamma$ and $\mu \in \bar U^+$, we define the
relative Green function $G_\mu(x,y; A)$ as the sum of $\mu$-probabilities of
paths from $x$ to $y$ that stay in $A$ except possibly at the first and last
step, i.e.,
\begin{equation*}
  G_\mu(x,y; A) = \sum_{n\geq 0} \sum_{\substack{x_0=x, x_1,\dotsc, x_n=y\\ x_1,\dotsc, x_{n-1}\in A}}
     \mu(x_0^{-1}x_1) \dotsm \mu(x_{n-1}^{-1}x_n).
\end{equation*}
The function $x\mapsto G_\mu(x,y;A)$ is harmonic on the set of points $x$
that are different from $y$ and can not jump outside of $A$ in one step. It
is bounded by $G_\mu(x,y)$. Hence, Proposition~\ref{prop:ineq_ancona_strong}
applies to relative Green functions on suitable domains. One has $G_\mu(x,y;
\Gamma)=G_\mu(x,y)$.

\subsection{The Martin kernel as a fixed point}

To prove that the entropy depends analytically on the measure,
using~\eqref{eq:h_integral}, we should show that the lifted Martin kernel
$K_\mu^B$ depends analytically on $\mu$. As we explained in the introduction,
our strategy is to exhibit $K_\mu^B$ as the fixed point of a suitable
operator, and conclude using the implicit function theorem.

We will now introduce this operator, first formally. Let $N>0$ be large
enough, only depending on the subset $F$ of $\Gamma$ in which the measures we
consider are supported. To $\xi \in
\partial_B\Gamma$, we associate a point $a(\xi)\in\Gamma$ with $\abs{a(\xi)}=N$ which,
heuristically, points in the direction of $\xi$. We require that
$h_\xi(a(\xi))=-N$ and that $a(\xi)$ only depends on the restriction of
$h_\xi$ to $B(e,N)$. We also associate to $\xi$ a set of points in $\Gamma$
denoted by $\Lambda(\xi)=\Lambda_0(\xi)$, which only depends on the
restriction of $h_\xi$ to $B(e,N)$ and contains the points ``far away from
$\xi$''. More precisely, let
\begin{equation*}
  \Lambda(\xi) = \{g\in \Gamma \st \gromprod{a(\xi)}{g}{e} \leq N/4\}.
\end{equation*}
If $N$ is large enough, $F$ and $F^{-1}$ are included in $\Lambda(\xi)$ for
every $\xi$. We will also need a slightly larger set
\begin{equation*}
  \Lambda'(\xi) = \{g\in \Gamma \st \gromprod{a(\xi)}{g}{e} \leq N/2\}.
\end{equation*}

Given $\xi \in \partial_B\Gamma$, there is now a canonical way to go towards
$\xi$ at infinity, starting from the identity $e$. First, we jump to
$a_1=a(\xi)$. Then, to $a_2=a_1 \cdot a( a_1^{-1}\xi)$. Then, to $a_3=a_2
\cdot a(a_2^{-1}\xi)$, and so on. To this process correspond nested sets
\begin{equation*}
  \Lambda(\xi) = \Lambda_0(\xi)\subset \Lambda_1(\xi)=a_1 \Lambda(a_1^{-1}\xi)
  \subset \Lambda_2(\xi) = a_2 \Lambda(a_2^{-1}\xi) \subset \dotsc
\end{equation*}
covering more and more the group. The successive boundaries of these sets
form a sequence of barriers between $e$ and $\xi$. Let also
$\Lambda'_k(\xi)=a_k \Lambda'(a_k^{-1} \xi)$. The complements of
$\Lambda_i(\xi)$ and $\Lambda'_i(\xi)$ are essentially horoballs centered at
$\pi_B \xi$, at distance respectively $iN+N/4$ and $iN+N/2$ of $e$.

Let $\mu \in \bar U^+$. Consider a nonnegative function $u$ on $\Gamma$ which
is, on $\Lambda'(\xi)$, positive, harmonic, and bounded by a finite linear
combination of functions $G_\mu(\cdot, y_i)$ with $y_i \notin \Lambda'(\xi)$.
It is classical that such a function satisfies
\begin{equation}
\label{eq:req_rec}
  u(x) = \sum_{y\notin \Lambda'(\xi)} G_\mu(x,y; \Lambda'(\xi)) u(y).
\end{equation}
Indeed, by harmonicity, $u(x) = \sum_y \mu(x^{-1}y) u(y)$. One can then apply
again this formula to all the $y$ which are still in $\Lambda'(\xi)$, and
repeat this algorithm up to time $n$. If $\mu$ is a probability measure, this
amounts to considering the random walk starting from $x$, stopped once it
exits $\Lambda'(\xi)$, and saying that the average value of $u$ along this
process does not change by harmonicity. We get a formula $u(x) = \sum_y
F_n(y) u(y)$, where for $y\notin \Lambda'(\xi)$
\begin{equation*}
  F_n(y)=\sum_{i\leq n} \sum_{\substack{x_0=x, x_1,\dotsc, x_i=y\\x_1,\dotsc, x_{i-1} \in \Lambda'(\xi)}}
    \mu(x_0^{-1}x_1) \dotsm \mu(x_{i-1}^{-1}x_i)
\end{equation*}
converges to $G_\mu(x,y; \Lambda'(\xi))$ when $n$ tends to $\infty$.  On the
other hand, for $y\in\Lambda'(\xi)$,
\begin{equation*}
  F_n(y) = \sum_{\substack{x_0=x, x_1,\dotsc, x_n=y\\x_1,\dotsc, x_{n-1}
    \in \Lambda'(\xi)}}
    \mu(x_0^{-1}x_1) \dotsm \mu(x_{n-1}^{-1}x_n).
\end{equation*}
One should show that the contribution of these points to the equality $u(x) =
\sum_y F_n(y) u(y)$ tends to $0$ when $n$ tends to $\infty$. This follows
from the domination condition $u(x) \leq C \sum G_\mu(x,y_i)$, since this
contribution is then bounded by the tails of the series defining the Green
function, which tend to $0$ as the series is finite.

\begin{rmk}
\label{rmk:Kmu_rec} The lifted Martin kernel $K_\mu^B(\cdot,\xi)=K_\mu(\cdot,
\pi_B \xi)$ is bounded on $\Lambda'(\xi)$ by $C G_\mu(\cdot, y_0)$ for any
$y_0 \notin \Lambda'(\xi)$, see Remark~\ref{rmq:Martin_borne}. Hence, it
satisfies~\eqref{eq:req_rec}.
\end{rmk}

In the formula~\eqref{eq:req_rec}, the points $y\notin \Lambda'(\xi)$ with a
nonzero coefficient $G_\mu(x,y; \Lambda'(\xi))$ are within bounded distance
of $\Lambda'(\xi)$, as the walk has bounded jumps. In particular, if $N$ is
large enough, they are all included in $\Lambda_1(\xi)$. Thus, this formula
can also be written as
\begin{equation*}
  u(x) = \sum_{y\in \Lambda_1(\xi)\setminus \Lambda'(\xi)} G_\mu(x,y; \Lambda'(\xi)) u(y).
\end{equation*}
This shows how the values of $u$ on $\Lambda_0(\xi)$ are determined by its
values on $\Lambda_1(\xi)$. This formula entails that harmonic functions are
fixed point of an operator, which has contracting properties thanks to the
good behavior of the kernel $G_\mu(x,y; \Lambda'(\xi))$ coming from Ancona
inequalities. Ultimately, this should provide a tractable analytic
characterization of the Martin kernel.

\medskip

We convert this heuristic discussion into a true operator $\boM_\mu$, defined
for any $\mu\in \bar U$. It acts on scalar-valued functions $f(x, \xi)$
defined on pairs $(x,\xi)$ with $\xi \in
\partial_B \Gamma$ and $x \in \Lambda(\xi)$, by the formula
\begin{equation}
\label{eq:defbommu}
  \boM_\mu f(x, \xi) = \sum_{y\in \Lambda_1(\xi) \setminus \Lambda'(\xi)}
     G_\mu(x,y; \Lambda'(\xi)) f(a(\xi)^{-1}y, a(\xi)^{-1}\xi).
\end{equation}
Note that the right-hand side is well defined as the point $(a(\xi)^{-1}y,
a(\xi)^{-1}\xi)$ belongs to the domain of definition of $f$, i.e.,
$a(\xi)^{-1}y \in \Lambda(a(\xi)^{-1}\xi)$, thanks to the condition $y\in
\Lambda_1(\xi)$.

To get a fixed point, we should projectivize this operator, normalizing for
instance so that the value at $(e,\xi)$ is always $1$. Hence, let
\begin{equation*}
  \boL_\mu f(x, \xi) = \boM_\mu f(x, \xi)/ \boM_\mu f(e, \xi).
\end{equation*}
This is not defined when $\boM_\mu f(e, \xi)=0$ for some $\xi$. We should
also ensure that the sums in the definition of $\boM_\mu$ are finite.

The contraction properties of this operator are central to the proof of
strong Ancona inequalities (Proposition~\ref{prop:ineq_ancona_strong}): For
$\mu\in\bar U^+$, these inequalities are obtained by letting this operator
act on a cone of positive functions, endowed with a Hilbert distance. To
apply the implicit function theorem, we will rather need contraction on a
Banach space, that we will deduce from
Proposition~\ref{prop:ineq_ancona_strong}. Hence, we will not reprove
Proposition~\ref{prop:ineq_ancona_strong}, but rather use its results to
obtain another form of contraction.

The iterates of $\boM_\mu$ (whence of $\boL_\mu$) have the same form. Indeed,
\begin{align*}
  \boM^2_\mu f(x, \xi)
  &= \sum_{y\notin \Lambda'(\xi)} G_\mu(x,y; \Lambda'(\xi)) \boM_\mu f(a(\xi)^{-1}y, a(\xi)^{-1}\xi)
  \\&=
  \sum_{y\notin \Lambda'(\xi)} G_\mu(x,y; \Lambda'(\xi)) \sum_{z\notin \Lambda'(a(\xi)^{-1}\xi)}
G_\mu(a(\xi)^{-1}y,z; \Lambda'(a(\xi)^{-1}\xi))
\times
\\& \hspace{3cm} \times
 f(a(a(\xi)^{-1}\xi)^{-1}z, a(a(\xi)^{-1}\xi)^{-1}a(\xi)^{-1}\xi).
\end{align*}
Let $w=a(\xi)z$, it belongs to the complement of $\Lambda'_1(\xi)$. Moreover,
\begin{equation*}
  G_\mu(a(\xi)^{-1}y,z;\Lambda'(a(\xi)^{-1}\xi))=G_\mu(y, w; \Lambda'_1(\xi)).
\end{equation*}
Decomposing a trajectory from $x\in \Lambda(\xi)$ to $w$ according to the
first point where it exits $\Lambda'(\xi)$, we have
\begin{equation*}
  \sum_{y\notin \Lambda'(\xi)} G_\mu(x,y; \Lambda'(\xi))G_\mu(y, w;
\Lambda'_1(\xi)) = G_\mu(x, w; \Lambda'_1(\xi)).
\end{equation*}
Recalling that $a_2(\xi) = a(\xi)\cdot a(a(\xi)^{-1}\xi)$, the above formula
for $\boM^2_\mu f(x, \xi)$ becomes
\begin{equation*}
  \boM^2_\mu f(x, \xi) = \sum_{w\notin \Lambda'_1(\xi)} G_\mu(x, w; \Lambda'_1(\xi)) f(a_2(\xi)^{-1}w, a_2(\xi)^{-1}\xi).
\end{equation*}
The only points with a nonzero coefficient $G_\mu(x, w; \Lambda'_1(\xi))$
belong to $\Lambda_2(\xi)$. Hence, the sum may be restricted to $w\in
\Lambda_2(\xi) \setminus \Lambda'_1(\xi)$.

In the same way, iterating this argument, one obtains
\begin{equation}
\label{eq:itereboMn}
  \boM^n_\mu f(x, \xi) = \sum_{y\in \Lambda_n(\xi) \setminus \Lambda'_{n-1}(\xi)}
     G_\mu(x, y; \Lambda'_{n-1}(\xi)) f(a_n(\xi)^{-1}y, a_n(\xi)^{-1}\xi).
\end{equation}
Finally, as the projectivization commutes with the iteration of $\boM_\mu$,
\begin{equation*}
  \boL_\mu^n f(x, \xi) = \boM^n_\mu f(x, \xi)/\boM^n_\mu f(e,\xi).
\end{equation*}

\bigskip

We define different norms on these functions. Let $\mu \in \bar U^+$. We set
\begin{equation*}
  \norm{f}_{\boC^0_\mu} = \sup_{\xi} \sup_{x\in \Lambda(\xi)} \abs{f(x, \xi)}/G_{\mu}(x,e)
\end{equation*}
and, for small $\beta>0$,
\begin{equation*}
  \norm{f}'_{\boC^\beta_\mu} = \sup_{d(\xi, \xi') \leq e^{-N}} d(\xi,\xi')^{-\beta}
  \sup_{x\in \Lambda(\xi)=\Lambda(\xi')} \abs{f(x, \xi)-f(x,\xi')}/G_{\mu}(x,e).
\end{equation*}
Finally, let $\norm{f}_{\boC^\beta_\mu} = \norm{f}_{\boC^0_\mu} +
\norm{f}'_{\boC^\beta_\mu}$. As these are H\"older-like norms, one checks
easily that these spaces are Banach algebras: if $f_1, f_2 \in
\boC^\beta_\mu$, then $f_1 f_2 \in \boC^\beta_\mu$ and
\begin{equation*}
  \norm{f_1 f_2}_{\boC^\beta_\mu} \leq \norm{f_1}_{\boC^\beta} \norm{f_2}_{\boC^\beta}.
\end{equation*}
The same holds in $\boC^0_\mu$.

\begin{rmk}
\label{rmq:equiv} Take $x\in \Lambda(\xi)$. Then a geodesics from $x$ to
$\pi_B \xi$ passes within uniformly bounded distance of $e$. By Ancona
inequalities, we obtain $G_\mu(x,y)/G_\mu(e,y) \asymp G_\mu(x,e)$ if $y$ is
close to $\pi_B \xi$. Letting $y$ tend to $\pi_B \xi$ yields
\begin{equation*}
  G_\mu(x,e) \asymp K_\mu^B(x,\xi) \quad\text{uniformly in }\xi\in \partial_B \Gamma \text{ and }x\in \Lambda(\xi).
\end{equation*}
Hence, in the definitions of the norms for $\boC^0_\mu$ and $\boC^\beta_\mu$,
we could have normalized by $K_\mu^B(x,\xi)$ instead of $G_\mu(x,e)$,
obtaining an equivalent norm.

This also implies that, if $d(\xi,\xi')\leq e^{-N}$ (ensuring that
$\Lambda(\xi)=\Lambda(\xi')$), then $K_\mu^B(x, \xi)\asymp K_\mu^B(x,\xi')$
uniformly on $x\in\Lambda(\xi)$.
\end{rmk}

If we want to study a fixed operator $\boL_\mu$ with $\mu\in \bar U^+$, the
space $\boC^\beta_\mu$ is most natural. However, we want to vary $\mu$ in
$U$. Hence, we need a reference space, independent of $\mu$. We recall that
we have a measure $\bar \mu$ which dominates all measures $\mu\in U$. The
spaces $\boC^\beta_{\mu}$ are all included in $\boC^\beta_{\bar\mu}$. Hence,
we can use the latter as a fixed reference space. We will simply write
$\boC^0 = \boC^0_{\bar\mu}$ and $\boC^\beta=\boC^\beta_{\bar\mu}$.

\begin{lem}
\label{lem:boMmuprime_bien_defini} For $\mu\in \bar U$ and small enough
$\beta$, the operator $\boM_{\mu}$ sends continuously $\boC^\beta$ into
itself. If $\mu\in \bar U^+$, it even sends $\boC^\beta$ into
$\boC^\beta_\mu$. Moreover, the map $(\mu,f)\mapsto \boM_{\mu} f$ is analytic
from $U \times \boC^\beta$ to $\boC^\beta$.
\end{lem}
\begin{proof}
Consider $f\in \boC^\beta$ and $\mu\in \bar U^+$. Then
\begin{align*}
  \abs{\boM_{\mu} f(x,\xi)}
  &\leq \sum_{y\in \Lambda_1(\xi) \setminus \Lambda'(\xi)} G_{\mu}(x,y;
     \Lambda'(\xi)) \abs{f(a(\xi)^{-1}y, a(\xi)^{-1}\xi)}
  \\&\leq \sum_{y\in \Lambda_1(\xi) \setminus \Lambda'(\xi)} G_{\mu}(x,y) G_{\bar\mu} (a(\xi)^{-1}y, e) \norm{f}_{\boC^0}.
\end{align*}
As $x\in \Lambda(\xi)$ and $y\notin \Lambda'(\xi)$, a geodesics from $x$ to
$y$ passes within bounded distance of $e$. Hence, by Ancona inequalities
(Lemma~\ref{lem:Ancona}),
\begin{equation*}
 G_{\mu}(x,y)\leq C G_{\mu}(x,e) G_{\mu}(e,y)\leq C G_\mu(x,e) G_{\bar\mu}(e,y).
\end{equation*}
Moreover, $G_{\bar\mu} (a(\xi)^{-1}y, e)=G_{\bar\mu}(y, a(\xi)) \leq C
G_{\bar\mu}(y,e)$ as $a(\xi)$ is a bounded distance away from $e$. Hence,
\begin{equation*}
  \abs{\boM_{\mu} f(x,\xi)} \leq C \sum_{y\in \Gamma} G_{\mu}(x,e) G_{\bar\mu}(e,y) G_{\bar\mu}(y,e)  \norm{f}_{\boC^0}.
\end{equation*}
Factorizing $G_{\mu}(x,e) \norm{f}_{\boC^0}$, we are left with the sum
$\sum_y G_{\bar\mu}(e,y) G_{\bar\mu}(y,e)$. Since the spectral radius of
$\bar\mu$ is $<1$ by construction, the map $r\mapsto G_{r\bar\mu}(e,e)$ is
well defined and analytic on a neighborhood of $1$. In particular,
$G_{\bar\mu}(e,e) + \partial G_{r\bar\mu}(e,e)/\partial r \restr_{r=1}$ is
finite. By an elementary computation (see~\cite[Proposition
1.9]{gouezel_lalley}), this is equal to $\sum_y G_{\bar\mu}(e,y)
G_{\bar\mu}(y,e)$. Hence, this sum is finite. This shows that
\begin{equation*}
  \norm{\boM_{\mu} f}_{\boC^0_\mu} \leq C \norm{f}_{\boC^0}.
\end{equation*}

Let us now control the H\"older norm. Consider $\xi, \xi'$, write
$d(\xi,\xi')=e^{-n}$. If $n\leq 2N$, then
\begin{align*}
  \abs*{\boM_\mu f(x,\xi) - \boM_\mu f(x,\xi')} / (d(\xi,\xi')^\beta G_\mu(x,e))
  &\leq e^{2N\beta} (\abs*{\boM_\mu f(x,\xi)} + \abs*{\boM_\mu f(x,\xi')})/G_\mu(x,e)
  \\&\leq C e^{2N\beta} \norm{f}_{\boC^0},
\end{align*}
by the sup norm control we have already proved. Assume now that $n>2N$. This
entails $\Lambda(\xi)=\Lambda(\xi')$, and $a(\xi)=a(\xi')$, and
$\Lambda(a(\xi)^{-1}\xi)=\Lambda(a(\xi')^{-1}\xi')$. Then, for $x\in
\Lambda(\xi)=\Lambda(\xi')$,
\begin{align*}
  \abs{\boM_{\mu}f(x,\xi)&-\boM_{\mu}f(x,\xi')}
  \\&= \abs*{\sum_{y\in\Lambda_1(\xi) \setminus \Lambda'(\xi)}
      G_{\mu}(x,y; \Lambda'(\xi)) (f(a(\xi)^{-1}y, a(\xi)^{-1}\xi)- f(a(\xi')^{-1}y, a(\xi')^{-1}\xi'))}
  \\&\leq \sum_{y\in\Lambda_1(\xi) \setminus \Lambda'(\xi)}
      G_{\mu} (x,y) d(a(\xi)^{-1}\xi, a(\xi)^{-1}\xi')^{\beta} G_{\bar\mu}(a(\xi)^{-1}y,e) \norm{f}'_{\boC^\beta}.
\end{align*}
As above, this is bounded by
\begin{equation*}
  C G_\mu(x,e) d(a(\xi)^{-1}\xi, a(\xi)^{-1}\xi')^{\beta} \norm{f}'_{\boC^\beta}
  \leq C e^{\beta N} G_\mu(x,e) d(\xi,\xi')^{\beta} \norm{f}'_{\boC^\beta},
\end{equation*}
by~\eqref{eq:stupid_contraction}. This shows that $\boM_{\mu}$ maps
continuously $\boC^\beta$ into $\boC^\beta_\mu$ when $\mu\in \bar U^+$.

When $\mu\in U$, we show in the same way that $\boM_{\mu}$ maps continuously
$\boC^\beta$ into itself. The computation is the same, excepted that $G_\mu$
is not positive any more. Hence, one should bound $\abs{G_\mu(x,y)}$ by
$G_{\bar\mu}(x,y)$ at the beginning of the computation, and then proceed in
the same way.

The map $F: (\mu,f)\mapsto \boM_\mu f$ on $\boC^\beta$ is linear in $f$. To
prove that it is analytic, we should control its dependence on $\mu$. Each
function $\mu\mapsto G_{\mu}(x,y; \Lambda'(\xi))$ is a limit of a degree $K$
polynomial, obtained by considering the weight of trajectories of length at
most $K$. Moreover, the corresponding $K$-truncated operators all satisfy the
same bounds as $F$, since the above computations apply. Hence, $F$ is the
uniform limit of analytic operators $F_K$. As analyticity is stable under
uniform convergence (see Paragraph~\ref{subsec:analytic}), it follows that
$F$ itself is analytic.
\end{proof}

\begin{cor}
\label{cor:boM_analytic} Let $\mu_1\in U$ and $f_1 \in \boC^\beta$ satisfy
$\boM_{\mu_1} f_1(e,\xi) \neq 0$ for all $\xi$. Then the operator
$(\mu,f)\mapsto \boL_{\mu} f$ is well defined and analytic from a
neighborhood of $(\mu_1, f_1)$ in $U\times \boC^\beta$, to $\boC^\beta$.
\end{cor}
\begin{proof}
The operator to be studied is the composition of the operators
$(\mu,f)\mapsto \boM_{\mu}f$ and $\boN:f\mapsto \tilde
f(x,\xi)=f(x,\xi)/f(e,\xi)$. The first one is well defined and analytic by
Lemma~\ref{lem:boMmuprime_bien_defini}. Hence, it suffices to show that
$\boN$ is well defined and analytic on the open set $\boD^\beta \subseteq
\boC^\beta$ of functions $f$ with $f(e,\xi)\neq 0$ for all $\xi$.

Let us first check that, if  $f\in \boD^\beta$, then $\boN f \in \boC^\beta$.
The supremum condition is obvious since $1/\abs{f(e,\xi)}$ is uniformly
bounded, by compactness of $\partial_B \Gamma$ and continuity. We check the
H\"older condition. We have
\begin{align*}
  \abs{\boN f(x, \xi)-\boN f(x,\xi')}
  &=\abs{f(x,\xi)/f(e,\xi) - f(x,\xi')/f(e,\xi')}
  \\&
  \leq C \abs{f(x,\xi) f(e,\xi') - f(e,\xi) f(x,\xi')}
  \\&
  \leq C \abs{f(x,\xi)-f(x,\xi')} \abs{f(e,\xi')} + C \abs{f(x,\xi')} \abs{f(e,\xi')-f(e,\xi)}.
\end{align*}
The first term is bounded by $C d(\xi,\xi')^\beta G_{\bar\mu}(x,e)$. In the
second one, we have $\abs{f(x,\xi')} \leq C G_{\bar\mu}(x,e)$ and
$\abs{f(e,\xi')-f(e,\xi)} \leq C d(\xi',\xi)^\beta$, hence we obtain the same
bound. This shows that $\boN f\in \boC^\beta$.

For the analyticity, let us fix $f\in \boD^\beta$. For small $h\in
\boC^\beta$,
\begin{equation*}
\boN (f+h) (x,\xi) = \frac{(f+h)(x,\xi)}{(f+h)(e,\xi)}
=\frac{(f+h)(x,\xi)}{f(e,\xi)} \sum (-1)^n (h/f)(e,\xi)^n.
\end{equation*}
This power series converges on a ball with positive radius, as $\boC^\beta$
is a Banach algebra.
\end{proof}

In the same way, one proves the following:
\begin{cor}
\label{cor:boMmu_analytic} Let $\mu\in \bar U^+$ and $f_1 \in \boC^\beta$
satisfy $\boM_{\mu} f_1(e,\xi) \neq 0$ for all $\xi$. Then the operator
$f\mapsto \boL_{\mu} f$ is well defined and analytic from a neighborhood of
$f_1$ in $ \boC^\beta$, to $\boC_{\mu}^\beta$.
\end{cor}
\begin{proof}
The operator $\boL_{\mu}$ is the composition of $\boM_{\mu}$ and $\boN:
f\mapsto \tilde f(x,\xi)=f(x,\xi)/f(e,\xi)$. The first one is linear from
$\boC^\beta$ to $\boC_{\mu}^\beta$ by Lemma~\ref{lem:boMmuprime_bien_defini},
the second one is analytic as we explained in the proof of
Corollary~\ref{cor:boM_analytic} (on $\boC^\beta$, but the same proof works
in $\boC^\beta_{\mu}$). Hence, their composition is analytic.
\end{proof}

\begin{lem}
\label{lem:Kmuprime_Cralpha} For small enough $\beta$ and for $\mu\in \bar
U^+$, the function $(x,\xi)\mapsto K_\mu^B(x,\xi)$ belongs to $\boC^\beta$.
\end{lem}
\begin{proof}
By~\eqref{rmq:equiv}, $K_\mu^B(x,\xi) \asymp G_{\mu}(x,e)$ uniformly in $\xi
\in \partial_B\Gamma$ and $x\in\Lambda(\xi)$. Moreover, $G_{\mu}(x,e) \leq
G_{\bar\mu}(x,e)$. This proves that $\norm{K_\mu^B}_{\boC^0}<\infty$.

Let us estimate its H\"older norm. Consider $\xi,\xi'$ with
$d(\xi,\xi')=e^{-n}$ for some $n\geq N$.
Proposition~\ref{prop:ineq_ancona_strong} applies to Martin kernels, by
Remark~\ref{rmq:Martin_borne}. It shows that, for any $x\in\Lambda(\xi)$,
\begin{equation*}
  \abs*{ \frac{K_\mu^B(x,\xi)/K_\mu^B(e,\xi)}{K_\mu^B(x,\xi')/K_\mu^B(e,\xi')}-1} \leq C e^{-C^{-1}n}.
\end{equation*}
As $K_\mu^B(e,\xi) = K_\mu^B(e,\xi')=1$, we obtain
\begin{equation*}
  \abs{K_\mu^B(x,\xi) - K_\mu^B(x,\xi')} \leq K_\mu^B(x,\xi') C e^{-C^{-1}n}.
\end{equation*}
The term on the right hand side is bounded by $C' G_{\bar\mu}(x,e)
d(\xi,\xi')^{\beta}$ if $\beta$ is small enough.
\end{proof}

The operators $\boL_\mu$ were defined precisely so that the following lemma
holds.
\begin{lem}
\label{lem:Kmu_pointfixe_unique} Let $\mu\in \bar U^+$. The function $(x,
\xi)\mapsto K_\mu^B(x, \xi)$ is a fixed point of $\boL_{\mu}$. It is the only
one among positive functions in $\boC^\beta$.
\end{lem}
\begin{proof}
As $K_\mu^B \in \boC^\beta$ by Lemma~\ref{lem:Kmuprime_Cralpha}, the operator
$\boM_{\mu}$ is well defined on $K_\mu^B$. By positivity, $\boL_{\mu}
K_\mu^B$ is also well defined.

By Remark~\ref{rmk:Kmu_rec}, the function $K_\mu^B$ satisfies the
equation~\eqref{eq:req_rec}, i.e.,
\begin{equation}
\label{eq:Kmuxig}
  K_\mu^B(x, \xi) = \sum_{y\notin \Lambda'(\xi)} G_{\mu}(x,y; \Lambda'(\xi)) K_\mu^B(y,\xi).
\end{equation}
It follows from its definition that $K_\mu^B$ satisfies the following
multiplicative cocycle equation:
\begin{equation*}
  K_\mu^B(y,\xi)=K_\mu^B(a^{-1}y, a^{-1}\xi)/K_\mu^B(a^{-1}, a^{-1}\xi).
\end{equation*}
Therefore,~\eqref{eq:Kmuxig} gives
\begin{equation*}
  K_\mu^B(x, \xi) = \sum_{y\notin \Lambda'(\xi)} G_\mu(x,y; \Lambda'(\xi))
    K_\mu^B(a(\xi)^{-1}y, a(\xi)^{-1}\xi)/K_\mu^B(a(\xi)^{-1}, a(\xi)^{-1}\xi).
\end{equation*}
Hence,
\begin{equation}
\label{eq:itere_Kmu}
  \boM_{\mu} K_\mu^B(x, \xi) = K_\mu^B(a(\xi)^{-1}, a(\xi)^{-1}\xi) K_\mu^B(x,\xi).
\end{equation}
Applying this equation to $x=e$ and dividing both sides, we get $K_\mu^B(x,
\xi)/K_\mu^B(e,\xi) = \boL_{\mu} K_\mu^B(x,\xi)$. As $K_\mu^B(e, \xi)=1$,
this concludes the proof that $K_\mu^B$ is a fixed point of $\boL_{\mu}$.

Consider now any fixed point $f>0$ of $\boL_\mu$ in $\boC^\beta$. Using the
equality~\eqref{eq:itereboMn}, we get
\begin{equation}
\label{eq:liuvcpouwxociv}
  f(x,\xi) = \boL_{\mu}^n f(x, \xi) = \boM_{\mu}^n f(e, \xi)^{-1} \cdot
    \sum_{\mathclap{y\notin \Lambda'_{n-1}(\xi)}} G_{\mu}(x, y; \Lambda'_{n-1}(\xi)) f(a_n(\xi)^{-1}y, a_n(\xi)^{-1}\xi).
\end{equation}
Proposition~\ref{prop:ineq_ancona_strong} yields, for all $x\in \Lambda(\xi)$
and all $y\notin \Lambda'_{n-1}(\xi)$ with $G_\mu(e,y; \Lambda'_{n-1}(\xi)) >
0$
\begin{equation*}
  \abs*{\frac{G_{\mu}(x, y; \Lambda'_{n-1}(\xi))/ G_{\mu}(e, y; \Lambda'_{n-1}(\xi))}{K_\mu^B(x,\xi)/K_\mu^B(e,\xi)} - 1}
  \leq Ce^{-C^{-1}n}.
\end{equation*}
Let $\epsilon>0$. For large enough $n$, we obtain
\begin{equation*}
  G_{\mu}(x, y; \Lambda'_{n-1}(\xi))
  = (1\pm \epsilon)G_{\mu}(e, y; \Lambda'_{n-1}(\xi))K_\mu^B(x,\xi).
\end{equation*}
Injecting this estimate into~\eqref{eq:liuvcpouwxociv} (and using the
nonnegativity of $f$), we get
\begin{align*}
  f(x,\xi) &= (1\pm \epsilon) \boM_{\mu}^n f(e, \xi)^{-1}  \cdot
      \sum_{\mathclap{y\notin \Lambda'_{n-1}(\xi)}} G_{\mu}(e, y; \Lambda'_{n-1}(\xi))K_\mu^B(x,\xi) f(a_n(\xi)^{-1}y, a_n(\xi)^{-1}\xi)
  \\& = (1\pm \epsilon) \boM_{\mu}^n f(e, \xi)^{-1}  \cdot K_\mu^B(x,\xi) \boM_{\mu}^n f(e, \xi)
  = (1\pm \epsilon) K_\mu^B(x,\xi).
\end{align*}
Letting $\epsilon$ tend to $0$, we obtain $f=K_\mu^B$.
\end{proof}

The following lemma encompasses the contraction properties of $\boL_\mu$ on
$\boC^\beta$. It is the main technical tool to be able to apply the implicit
function theorem later on.

\begin{lem}
\label{lem:diff_contracte} Assume that $\beta$ is small enough. Let $\mu\in
\bar U^+$. The differential of $\boL_\mu$ at $K_\mu^B$ satisfies
\begin{equation*}
 \norm{ D \boL_\mu^n(K_\mu^B) : \boC^\beta \to \boC^\beta} \leq C e^{-\rho n},
\end{equation*}
for some $\rho>0$, some $C>0$ and all $n\geq 0$.
\end{lem}
\begin{proof}
Denoting by $I$ the inclusion of $\boC^\beta_\mu$ in $\boC^\beta$, we have $D
\boL_\mu^{n+1}(K_\mu^B) = I \circ D\boL_\mu^{n}(K_\mu^B) \circ
D\boL_\mu(K_\mu^B)$ where the operator on the right is well defined and maps
continuously $\boC^\beta$ into $\boC^\beta_\mu$ by
Corollary~\ref{cor:boMmu_analytic}. Thus, it suffices to show that
\begin{equation}
\label{eq:derivee_controlee}
  \norm{D \boL_\mu^n(K_\mu^B) : \boC^\beta_\mu \to \boC^\beta_\mu} \leq C e^{-\rho n}.
\end{equation}

For ease of notations, we will write $\bfx$ for a pair $(x,\xi)$, and
$e_\bfx=(e,\xi)$. By Remark~\ref{rmq:equiv}, we can define equivalent norms
on $\boC^0_\mu$ and $\boC^\beta_\mu$ by
\begin{equation*}
  \norm{f}_{\bar \boC^0_\mu} = \sup_{\bfx} \abs{f(\bfx)}/K_\mu^B(\bfx),
\end{equation*}
and
\begin{equation*}
  \norm{f}'_{\bar \boC^\beta_\mu} = \sup_{d(\xi,\xi') \leq e^{-N}}
  \sup_{x\in\Lambda(\xi)}
d(\xi,\xi')^{-\beta} \abs{f(x,\xi)-f(x,\xi')}/K_\mu^B(x,\xi).
\end{equation*}
In this last equation, we could have use $K_\mu^B(x,\xi)$ or
$K_\mu^B(x,\xi')$ since their ratio is bounded, again by
Remark~\ref{rmq:equiv}. It suffices to prove the
inequality~\eqref{eq:derivee_controlee} for the norm $\bar \boC^\beta_\mu$,
which is equivalent to the original one.

We have $\boL_\mu^n f(\bfx) = \boM_\mu^n f(\bfx)/\boM_\mu^n f(e_\bfx)$.
Consequently,
\begin{equation}
\label{eq:formule_derivee}
  D\boL_\mu^n(K_\mu^B) f (\bfx) = \frac
    {\boM_\mu^n f(\bfx) \cdot \boM_\mu^n K_\mu^B(e_\bfx) - \boM_\mu^n f(e_\bfx) \boM_\mu^n K_\mu^B(\bfx)}
    { (\boM_\mu^n K_\mu^B(e_\bfx))^2}.
\end{equation}

By~\eqref{eq:itereboMn}, the operator $\boM_\mu^n$ can be written as
\begin{equation}
\label{eq:express_kernel}
  \boM_\mu^n f(\bfx) = \sum_\bfy F_n(\bfx,\bfy) f(\bfy)
\end{equation}
where the kernel $F_n$ is a relative Green function, given by
\begin{equation*}
  F_n((x,\xi), (y,\eta))=1_{\eta=a_n(\xi)^{-1}\xi}
    1_{a_n(\xi)y \in \Lambda_n(\xi) \setminus \Lambda'_{n-1}(\xi)}G_\mu(x, a_n(\xi)y; \Lambda'_{n-1}(\xi)).
\end{equation*}
When we write $F_n((x,\xi), (y,\eta))$, we will implicitly only consider
those pairs where $F_n$ can be nonzero, i.e., those where
$\eta=a_n(\xi)^{-1}\xi$ and $a_n(\xi)y \in \Lambda_n(\xi) \setminus
\Lambda'_{n-1}(\xi)$.

The following property of $F_n$ follows from
Proposition~\ref{prop:ineq_ancona_strong}: There exist $\rho_0>0$ and $C>0$
such that, for all $\bfx$, and all $\bfy, \bfz$ with $F_n(\bfx,\bfy)\neq 0$,
$F_n(\bfx, \bfz) \neq 0$,
\begin{equation*}
  \abs*{\frac{ F_n(\bfx,\bfz)/F_n(e_{\bfx},\bfz)}{F_n(\bfx,\bfy)/F_n(e_\bfx,\bfy)}-1} \leq C e^{-\rho_0 n}.
\end{equation*}
In particular,
\begin{equation}
\label{eq:ineq_Fn_produit}
  \abs{ F_n(e_\bfx,\bfy) F_n(\bfx,\bfz) - F_n(\bfx,\bfy) F_n(e_\bfx, \bfz)}
    \leq C e^{-\rho_0 n} F_n(\bfx,\bfy) F_n(e_{\bfx},\bfz).
\end{equation}

By~\eqref{eq:formule_derivee},
\begin{equation}
\label{eq:lmkjqmlksdjf}
  D\boL_\mu^n(K_\mu^B) f (\bfx) = \frac
  { \sum_{\bfy,\bfz} (F_n(\bfx,\bfy) F_n(e_\bfx,\bfz)- F_n(e_\bfx,\bfy) F_n(\bfx,\bfz))f(\bfy) K_\mu^B(\bfz)}
  { \paren*{ \sum_\bfy F_n(e_\bfx, \bfy) K_\mu^B(\bfy)}^2}.
\end{equation}
With~\eqref{eq:ineq_Fn_produit}, we get
\begin{equation*}
  \abs{D\boL_\mu^n(K_\mu^B) f(\bfx)} \leq C e^{-\rho_0 n} \frac
  {\sum_{\bfy,\bfz} F_n(\bfx,\bfy)F_n(e_\bfx,\bfz) \abs{f(\bfy)} K_\mu^B(\bfz)}
  {\paren*{ \sum_\bfy F_n(e_\bfx,\bfy) K_\mu^B(\bfy)}^2}.
\end{equation*}
The sum on the numerator can be factored. The part with $\bfz$ cancels out
one of the factors of the denominator. In the $\bfy$ part, we have by
definition $\abs{f(\bfy)} \leq \norm{f}_{\bar \boC^0_\mu} K_\mu^B(\bfy)$.
Hence,
\begin{align*}
  \abs{D\boL_\mu^n(K_\mu^B) f (\bfx)} &
  \leq C \norm{f}_{\bar \boC^0_\mu} e^{-\rho_0 n} \frac
  {\sum_\bfy F_n(\bfx,\bfy) K_\mu^B(\bfy)}{\sum_\bfy F_n(e_\bfx, \bfy) K_\mu^B(\bfy)}
  = C \norm{f}_{\bar \boC^0_\mu} e^{-\rho_0 n} \boL_\mu^n K_\mu^B(\bfx)
  \\&
  = C \norm{f}_{\bar \boC^0_\mu} e^{-\rho_0 n} K_\mu^B(\bfx),
\end{align*}
as $K_\mu^B$ is a fixed point of $\boL_\mu$. This shows that
$\norm{D\boL_\mu^n(K_\mu^B) f}_{\bar \boC^0_\mu} \leq C e^{-\rho_0 n}
\norm{f}_{\bar \boC^0_\mu}$, as claimed.

Let us now control its H\"older norm. Let $\bfx=(x,\xi)$ and
$\bfx'=(x,\xi')$, with $d(\xi,\xi')=e^{-k}$ for some $k\geq N$. We should
bound $\abs{D\boL_\mu^n(K_\mu^B) f (\bfx) - D\boL_\mu^n(K_\mu^B) f (\bfx')}/(
d(\xi,\xi')^{\beta}K_\mu^B(\bfx))$.

If $k\leq (n+1)N$, we write
\begin{align*}
  \abs{D\boL_\mu^n(K_\mu^B) f (\bfx) - D\boL_\mu^n(K_\mu^B) f (\bfx')}
  &\leq \abs{D\boL_\mu^n(K_\mu^B) f (\bfx)} + \abs{D\boL_\mu^n(K_\mu^B) f (\bfx')}
  \\&\leq \norm{D\boL_\mu^n(K_\mu^B) f}_{\bar \boC^0_\mu} (K_\mu^B(\bfx)+K_\mu^B(\bfx'))
  \\&\leq C \norm{D\boL_\mu^n(K_\mu^B) f}_{\bar \boC^0_\mu} K_\mu^B(\bfx),
\end{align*}
where the last inequality uses the fact that $K_\mu^B(\bfx') \leq C
K_\mu^B(\bfx)$, explained in Remark~\ref{rmq:equiv}. Therefore,
\begin{align*}
  \abs{D\boL_\mu^n(K_\mu^B) f (\bfx) - D\boL_\mu^n(K_\mu^B) f (\bfx')} / (d(\xi,\xi')^{\beta}K_\mu^B(\bfx))
  &\leq C e^{\beta (n+1) N} \norm{D\boL_\mu^n(K_\mu^B) f}_{\bar \boC^0_\mu}
  \\&\leq C e^{\beta (n+1) N} e^{-\rho_0 n} \norm{f}_{\bar \boC^0_\mu}.
\end{align*}
If $\beta$ is small enough (i.e., $\beta < \rho_0/N$), this is exponentially
small as desired.

Assume now that $k>(n+1)N$. Then $a_n(\xi)=a_n(\xi')$ (we will simply denote
it by $a_n$). Moreover, the two horofunctions given by $\eta= a_n^{-1}\xi$
and $\eta'=a_n^{-1}\xi'$ coincide on the ball of radius $N$. Thus, the
summation set in~\eqref{eq:express_kernel} is the same for $\bfx$ and
$\bfx'$. Moreover, for any $y$, we have $F_n((x,\xi), (y,\eta))=F_n((x,\xi'),
(y,\eta'))$. Taking $\xi$, $\xi'$ (and therefore $\eta$ and $\eta'$) as
fixed, we will simply write $\tilde F_n(x,y)$ for this common quantity. Let
$u(\eta) = 1/(\sum_y \tilde F_n(e, y) K_\mu^B(y,\eta))^2$, and define
$u(\eta')$ in the same way with $\eta'$. By~\eqref{eq:lmkjqmlksdjf},
\begin{equation*}
  D\boL_\mu^n(K_\mu^B) f (\bfx)
  = \sum_{y,z} (\tilde F_n(x,y) \tilde F_n(e,z)- \tilde F_n(e,y) \tilde F_n(x,z))f(y, \eta) K_\mu^B(z, \eta) u(\eta).
\end{equation*}
The same formula holds for $\bfx'$. Hence,
\begin{multline}
\label{eq:ineq_diff}
  \abs*{D\boL_\mu^n(K_\mu^B) f (\bfx) - D\boL_\mu^n(K_\mu^B) f (\bfx')} =
  \\
  \abs*{\sum_{y,z} \paren[\big]{\tilde F_n(x,y) \tilde F_n(e,z)- \tilde F_n(e,y) \tilde F_n(x,z)}
      \!\cdot\! \paren[big]{f(y, \eta) K_\mu^B(z, \eta) u(\eta) - f(y, \eta')K_\mu^B(z, \eta') u(\eta')}}.
\end{multline}

We use the equality $abc-a'b'c' = (a-a')bc+a'(b-b')c+a'b'(c-c')$ to bound the
last difference. We have
\begin{equation*}
  \abs{f(y, \eta)-f(y, \eta')} \leq d(\eta,\eta')^{\beta} \norm{f}_{\bar
\boC^\beta_\mu} K_\mu^B(y,\eta).
\end{equation*}
As $K_\mu^B \in \boC^\beta_\mu = \bar \boC^\beta_\mu$ by
Lemma~\ref{lem:Kmuprime_Cralpha},
\begin{equation*}
  \abs{K_\mu^B(z, \eta)- K_\mu^B(z, \eta')} \leq C d(\eta, \eta')^{\beta} K_\mu^B(z, \eta).
\end{equation*}
Finally, $u(\eta)=1/(\boM_\mu^n K_\mu^B(e, \xi))^2$. By~\eqref{eq:itere_Kmu},
this equals $1/K_\mu^B(a_n^{-1}, a_n^{-1}\xi)^2 = 1/K_\mu^B(a_n^{-1},
\eta)^2$. We have $\abs{ K_\mu^B(a_n^{-1}, \eta) - K_\mu^B(a_n^{-1}, \eta')}
\leq C d(\eta,\eta')^\beta K_\mu^B(a_n^{-1}, \eta)$ since $K_\mu^B\in
\boC^\beta_\mu = \bar \boC^\beta_\mu$. Moreover, the ratio $K_\mu^B(a_n^{-1},
\eta)/K_\mu^B(a_n^{-1}, \eta')$ is uniformly bounded by
Remark~\ref{rmq:equiv}. Therefore,
\begin{align*}
  \abs{u(\eta)-u(\eta')}&
  =\abs*{\frac
  {(K_\mu^B(a_n^{-1}, \eta) - K_\mu^B(a_n^{-1}, \eta'))(K_\mu^B(a_n^{-1}, \eta) + K_\mu^B(a_n^{-1}, \eta'))}
  {K_\mu^B(a_n^{-1}, \eta)^2 K_\mu^B(a_n^{-1}, \eta')^2}
  }
  \\&
  \leq C d(\eta,\eta')^\beta / K_\mu^B(a_n^{-1}, \eta)^2
  = C d(\eta, \eta')^\beta u(\eta).
\end{align*}
Combining these inequalities, we obtain
\begin{equation*}
  \abs{f(y, \eta) K_\mu^B(z, \eta) u(\eta) - f(y, \eta')K_\mu^B(z, \eta') u(\eta')}
  \leq C d(\eta,\eta')^\beta \norm{f}_{\bar \boC^\beta_\mu}
  K_\mu^B(y, \eta) K_\mu^B(z, \eta) u(\eta).
\end{equation*}
Together with~\eqref{eq:ineq_diff}, this yields
\begin{multline*}
  \abs*{D\boL_\mu^n(K_\mu^B) f (\bfx) - D\boL_\mu^n(K_\mu^B) f (\bfx')}
  \\
  \leq C d(\eta, \eta')^\beta \norm{f}_{\bar \boC^\beta_\mu}
      \sum_{y,z} \abs[\big]{\tilde F_n(x,y) \tilde F_n(e,z)- \tilde F_n(e,y) \tilde F_n(x,z)}
      K_\mu^B(y, \eta) K_\mu^B(z, \eta) u(\eta).
\end{multline*}
The sum is precisely the sum we have handled in the control of the sup norm,
in~\eqref{eq:lmkjqmlksdjf}, with $f$ replaced by $K_\mu^B$. We have shown
that it is bounded by $C e^{-\rho_0 n} K_\mu^B(\bfx)$. Finally, we get
\begin{equation*}
  \abs*{D\boL_\mu^n(K_\mu^B) f (\bfx) - D\boL_\mu^n(K_\mu^B) f (\bfx')} \leq
  C e^{-\rho_0 n} d(\eta,\eta')^\beta K_\mu^B(\bfx) \norm{f}_{\bar \boC^\beta_\mu}.
\end{equation*}
As $\eta = a_n^{-1} \xi$ and $\eta'=a_n^{-1}\xi'$ with $\abs{a_n}\leq nN$, we
have $d(\eta,\eta') \leq e^{nN} d(\xi,\xi')$
by~\eqref{eq:stupid_contraction}. Therefore,
\begin{equation*}
  \abs*{D\boL_\mu^n(K_\mu^B) f (\bfx) - D\boL_\mu^n(K_\mu^B) f (\bfx')} \leq
  C e^{\beta n N} e^{-\rho_0 n} d(\xi,\xi')^\beta K_\mu^B(\bfx) \norm{f}_{\bar \boC^\beta_\mu}.
\end{equation*}
If $\beta<\rho_0/N$, this is again exponentially small as desired.
\end{proof}

Let us consider the transformation $\boQ : (\mu, f) \mapsto \boL_\mu f - f$,
defined on a neighborhood of $(\mu_0, K_{\mu_0}^B)$ in $U \times \boC^\beta$
for a suitably small $\beta$. It satisfies $\boQ(\mu_0, K_{\mu_0}^B)=0$ as
$K_{\mu_0}^B$ is a fixed point of $\boL_{\mu_0}$ by
Lemma~\ref{lem:Kmu_pointfixe_unique}. Moreover, $\partial_f \boQ(\mu_0,
K_{\mu_0}^B) = D \boL_{\mu_0}(K_{\mu_0}^B) - \Id$ is invertible, by
Lemma~\ref{lem:diff_contracte}. The implicit function theorem, in its
analytic version, shows the existence of an analytic family $\mu\mapsto
f_\mu$, for $\mu$ close to $\mu_0$, with $f_{\mu_0}=K_{\mu_0}^B$, and such
that $\boQ (\mu, f_\mu)=0$, i.e., $\boL_\mu f_\mu = f_\mu$.

\begin{lem}
\label{lem:fmu_positive} For $\mu\in U^+$ close enough to $\mu_0$, the
function $f_\mu$ is everywhere positive.
\end{lem}
\begin{proof}
As $\mu\mapsto f_\mu$ is analytic, it is continuous. For $\mu$ close to
$\mu_0$, the function $f_\mu$ is close in $\boC^\beta$ to $f_0=K_{\mu_0}^B$.
As
\begin{equation*}
  f_\mu(x,\xi) = \boL_\mu f_\mu(x,\xi) = \boM_\mu f_\mu(x,\xi) / \boM_\mu
  f_\mu(e,\xi),
\end{equation*}
it suffices to show that $\boM_\mu f$ is positive if $f$ is close enough to
$f_0$ in $\boC^\beta$, and $\mu\in U^+$.

Let $\epsilon>0$. Any $f$ close enough to $f_0$ in $\boC^\beta$ satisfies
$f(x,\xi)\geq f_0(x,\xi)  - \epsilon G_{\bar \mu}(x,e)$ by definition of the
norm. We obtain
\begin{equation}
\label{eq:low_bound}
  \boM_\mu f(x,\xi) \geq \boM_{\mu} f_0(x,\xi) -\epsilon \boM_{\mu} G_{\bar\mu}(x,e)
  \geq \boM_{\mu} f_0(x,\xi) - C\epsilon G_\mu(x,e).
\end{equation}
For the last inequality, we used the fact that $(x,\xi)\mapsto
G_{\bar\mu}(x,e)$ belongs obviously to $\boC^0$, and the fact that $\boM_\mu$
maps $\boC^0$ into $\boC^0_\mu$ by Lemma~\ref{lem:boMmuprime_bien_defini}.

Let $\xi\in \partial_B \Gamma$. Consider a point $y=y(\xi)$, close to a
geodesic from $e$ to $a(\xi)$, which does not belong to $\Lambda'(\xi)$ but
such that $G_{\mu_0}(e, y; \Lambda'(\xi))>0$, i.e., $y$ is close enough to
the boundary of $\Lambda'(\xi)$ to be an exit point of the random walk in
$\Lambda'(\xi)$. It satisfies $G_\mu(e,y;\Lambda'(\xi))>0$ as $\mu\geq
\mu_0/2$ thanks to the definition of $U$. The function $x\mapsto G_\mu(x,y;
\Lambda'(\xi))$ is positive and harmonic on $\Lambda(\xi)$. Therefore,
Proposition~\ref{prop:ineq_ancona_strong} applied to this function and to
$x\mapsto G_\mu(x,y)$ shows that
\begin{equation*}
  \frac{G_\mu(x,y)/G_\mu(e,y)}{G_\mu(x,y; \Lambda'(\xi)) / G_\mu(e,y; \Lambda'(\xi))} \leq C,
\end{equation*}
uniformly in $x\in\Lambda(\xi)$. The quantities $G_\mu(e,y; \Lambda'(\xi))$
and $G_\mu(e,y)$ are bounded from above and from below, uniformly. We obtain
$G_\mu(x,y; \Lambda'(\xi)) \geq C^{-1} G_\mu(x,y)$, which is $\geq C^{-1}
G_\mu(x,e)$ thanks to Harnack inequalities (as $d(e,y)$ is uniformly
bounded).

By the definition~\eqref{eq:defbommu} of $\boM_\mu$,
\begin{equation*}
  \boM_{\mu} f_0(x,\xi) \geq G_\mu(x,y(\xi); \Lambda'(\xi)) f_0(a(\xi)^{-1}y, a(\xi)^{-1}\xi)
  \geq C^{-1} G_\mu(x,e) \cdot C^{-1},
\end{equation*}
as $f_0$ is uniformly bounded from below on points which are a bounded
distance away from the identify.

Finally, we obtain from~\eqref{eq:low_bound}
\begin{equation*}
  \boM_\mu f(x,\xi) \geq C^{-1} G_\mu(x,e) - C\epsilon G_\mu(x,e).
\end{equation*}
If $\epsilon$ is small enough, this is bounded from below by a positive
multiple of $G_\mu(x,e)$, uniformly in $(x,\xi)$.
\end{proof}

We can now prove that the entropy depends analytically on the measure.

\begin{proof}[Proof of Theorem~\ref{thm:entropy_analytic}]
By Lemma~\ref{lem:fmu_positive}, there exists a neighborhood $V_0$ of $\mu_0$
in $U\subseteq \boP(F)$ such that, for $\mu \in V_0 \cap U^+$, the function
$f_\mu$ is a fixed point of $\boL_\mu$ in $\boC^\beta$, everywhere positive.
By Lemma~\ref{lem:Kmu_pointfixe_unique}, it has to coincide with $K_\mu^B$.
Moreover, $\mu\mapsto f_\mu$ is analytic from $V_0$ to $\boC^{\beta}$ for
some $\beta>0$.

We have also constructed in Theorem~\ref{thm:stationary_analytic} an analytic
map $\mu \mapsto \Phi(\mu)$ on a neighborhood $V_1$ of $\mu_0$, taking values
in $(C^\beta)^*$, which corresponds for $\mu \in \boP_1^+(F)$ to the
integration against a $\mu$-stationary measure $\nu_\mu$ on
$\partial_B\Gamma$.

Let $V=V_0\cap V_1$. For $\mu \in V\cap \boP_1^+$, the integral
expression~\eqref{eq:h_integral} of the entropy gives
\begin{align*}
  h(\mu) & = -\sum_{x\in F} \mu(x) \int_{\partial_B \Gamma} \log K_\mu^B(x^{-1},\xi) \dd\nu_\mu(\xi)
  \\&
  = - \sum_{x\in F} \mu(x) \cdot \Phi(\mu)( \xi \mapsto \log f_\mu(x^{-1},\xi)).
\end{align*}
The expression on the second line is well defined and analytic on $V$, giving
the desired analytic extension of the entropy. Indeed, by definition of
$\boC^\beta$, each function $\xi \mapsto f_\mu(x^{-1},\xi)$ (with $x$ fixed)
belongs to $C^\beta$ and depends analytically on $\mu$. Composing with the
logarithm, applying the analytic linear form $\Phi(\mu)$, and doing a finite
summation over $F$, everything remains analytic.
\end{proof}

\section{Central limit theorem}

\label{sec:CLT}

In this section, we prove the central limit theorem, Theorem~\ref{thm:CLT}.
Once Proposition~\ref{prop:spectral_description} is available, it follows
using the method of Le Page~\cite{lepage_trick} as described for instance
in~\cite{bougerol_lacroix} or~\cite{benoist_quint_livre}. Although the
details are now standard, we will sketch the argument especially since the
lack of true contraction on $\partial'_B\Gamma$ creates possible periodicity
problems. We will prove the result for measures with an exponential moment of
order $\alpha>0$, as announced in Remark~\ref{rmk:infinite_support}.

For $\mu\in \boP_1^+(\alpha)$, we start from the operator $L_\mu$, acting on
$C^\beta(\partial'_B \Gamma)$ for some small enough $\beta>0$ by the formula
$L_\mu u(\xi) = \sum_g \mu(g) u(g\xi)$. By
Proposition~\ref{prop:spectral_description}, it has a simple eigenvalue at
$1$. For $t\in \R$, define a perturbed operator $L_{\mu,t}$ by
\begin{equation*}
  L_{\mu, t} u(\xi) = \sum_{g\in \Gamma}\mu(g) e^{\ic t c_B(g, \xi)} u(g\xi),
\end{equation*}
where $c_B$ is the Busemann cocycle~\eqref{eq:def_cB}. Then
\begin{equation*}
  L_{\mu,t}^2 u(\xi) = \sum_{g_1, g_2}\mu(g_1) \mu(g_2) e^{\ic t c_B(g_1, g_2 \xi)}e^{\ic t c_B(g_2, \xi)} u(g_1 g_2 \xi)
  =\sum_g \mu^{*2}(g) e^{\ic t c_B(g,\xi)} u(g\xi),
\end{equation*}
thanks to the cocycle equation~\eqref{eq:cocycle_cB} for $c_B$. Iterating
this argument, one gets
\begin{equation*}
  L_{\mu,t}^n u(\xi) = \sum_g \mu^{*n}(g) e^{\ic t c_B(g,\xi)} u(g\xi).
\end{equation*}
In particular, if $Z_n$ denotes the right random walk driven by $\mu$, the
characteristic function of $c_B(Z_n,\xi)$ can be expressed for all
$\xi\in\partial'_B\Gamma$ as follows:
\begin{equation}
\label{eq:express_char}
  \Ebb( e^{\ic t c_B(Z_n,\xi)}) = L_{\mu,t}^n 1(\xi).
\end{equation}

Fix once and for all a point $\xi \in \partial'_B\Gamma$. We will prove a
central limit theorem for $c_B(\cdot,\xi)$, i.e., prove that $(c_B(Z_n,\xi)-
n\ell)/\sqrt{n}$ converges in distribution to a Gaussian random variable. By
L\'evy's theorem, it suffices to prove the pointwise convergence of the
characteristic functions. It will be obtained from~\eqref{eq:express_char}
and the good spectral properties of $L_{\mu,t}$.

One checks easily that the map $t\mapsto L_{\mu,t}$ is analytic. By
Proposition~\ref{prop:spectral_description}, the operator $L_{\mu,0}=L_\mu$
has a simple eigenvalue at $1$, finitely many additional simple eigenvalues
$\rho$ of modulus $1$, and the rest of its spectrum is contained in a disk of
radius $<1$. This spectral picture persists for small $t$, see for
instance~\cite{kato_pe}: there are spectral projections $\Pi_{1,t}$ and
$\Pi_{\rho,t}$ and $\Pi_{<1,t}$, all depending analytically on $t$, such that
\begin{equation*}
  L_{\mu,t}^n = L_{\mu,t}^n \Pi_{1,t} + \sum_{\rho\not=1}L_{\mu,t}^n \Pi_{\rho,t}
  + L_{\mu,t}^n \Pi_{<1,t}.
\end{equation*}
(All these projections also depend on $\mu$, we suppress it from the
notations for convenience.) Moreover, $\Pi_{1,t}$ is a one-dimensional
projection, and $L_{\mu,t}$ acts on its image as the multiplication by the
corresponding eigenvalue $\lambda(t)=\lambda_\mu(t)$. This eigenvalue also
depends analytically on $t$.

Denote by $\ell$ the escape rate of the random walk.
Then~\eqref{eq:express_char} implies that, for any $t$, and for large enough
$n$ (so that $t/\sqrt{n}$ is in the neighborhood of $0$ where the above
spectral description holds true)
\begin{multline*}
  \Ebb\paren*{e^{\ic t \frac{c_B(Z_n,\xi)-n\ell}{\sqrt{n}}}}
  =e^{-\ic t \sqrt{n} \ell} L_{\mu, t/\sqrt{n}}^n 1(\xi)
  \\=e^{-\ic t \sqrt{n} \ell} \left(
    \lambda(t/\sqrt{n})^n \Pi_{1,t/\sqrt{n}}1(\xi) + \sum_{\rho\not=1}L_{\mu,t/\sqrt{n}}^n \Pi_{\rho,t/\sqrt{n}}1(\xi)
  + L_{\mu,t/\sqrt{n}}^n \Pi_{<1,t/\sqrt{n}}1(\xi)\right).
\end{multline*}
Since the function $1$ belongs to the eigenspace for the eigenvalue $1$, one
has $\Pi_{\rho,0}1=\Pi_{<1, 0}1=0$. Hence, when $n$ tends to infinity, the
functions $\Pi_{\rho,t/\sqrt{n}}1$ and $\Pi_{<1,t/\sqrt{n}}1$ tend to $0$ in
$C^\beta$, and therefore in $C^0$. As the iterates $L_{\mu,t/\sqrt{n}}^n$ are
bounded in sup norm by $1$, it follows that $L_{\mu,t/\sqrt{n}}^n
\Pi_{\rho,t/\sqrt{n}}1(\xi)$ and $L_{\mu,t/\sqrt{n}}^n
\Pi_{<1,t/\sqrt{n}}1(\xi)$ tend to $0$.

In the same way, $\Pi_{1,t/\sqrt{n}}1$ tends to $\Pi_{1,0}1=1$. Finally, we
obtain
\begin{equation}
\label{eq:converge_b}
  \Ebb\paren*{e^{\ic t \frac{c_B(Z_n,\xi)-n\ell}{\sqrt{n}}}}
  =e^{-\ic t \sqrt{n} \ell} \lambda(t/\sqrt{n})^n + o(1).
\end{equation}

As $\lambda$ is an analytic function with $\lambda(0)=1$, it has a Taylor
expansion at $0$. It is convenient to write it as $\lambda(t)=e^{\ic at -b
t^2/2 + o(t^2)}$, with $a=-\ic \lambda'(0)$ and $b=-a^2-\lambda''(0)$.

>From this point on, two approaches are possible:
\begin{itemize}
\item One can proceed directly, without identifying $a$ and $b$.
\item Or one can identify $a$ and $b$ by a tedious spectral computation.
\end{itemize}
The second approach is more precise (using it, one can prove that the
variance in the central limit theorem is positive). However, since the first
one is more illuminating, we will argue first in this way.

Reproducing the above computation but for $c_B/n$, one has
\begin{equation*}
  \Ebb(e^{\ic t c_B(Z_n, \xi)/n}) = \lambda(t/n)^n+o(1) \to e^{\ic a t}.
\end{equation*}
By definition, $c_B(Z_n, \xi)=h_\xi(Z_n^{-1})$. The point $Z_n^{-1}$ is
distributed as the position $W_n$ at time $n$ of the reverse random walk, for
the measure $\check\mu$ given by $\check\mu(x)=\mu(x^{-1})$. As $\check\mu$
is also admissible, this reverse random walk converges almost surely to a
point on the Gromov boundary. As its exit distribution has no atom, this
limit is almost surely different from $\pi_B(\xi)$, the projection of $\xi$
in the Gromov boundary. Then $h_\xi(W_n)-d(e,W_n)$ remains bounded almost
surely by~\eqref{eq:dh_bounded_diff}. Moreover, $d(e,W_n)/n$ tends to $\ell$,
hence $h_\xi(W_n)/n$ tends almost surely to $\ell$. This implies that
\begin{equation*}
  c_B(Z_n,\xi)/n=h_\xi(Z_n^{-1})/n\text{ tends in probability to $\ell$}.
\end{equation*}
Therefore, $\Ebb(e^{\ic t c_B(Z_n, \xi)/n})$ converges to $e^{\ic t \ell}$.
Hence, $a=\ell$.

For future use, note that the same argument shows that, for any sequence
$r_n$ tending to infinity,
\begin{equation}
\label{eq:cBproba2}
  \frac{c_B(Z_n,\xi)- d(e,Z_n)}{r_n} \text{ tends in probability to $0$}.
\end{equation}

Using $a=\ell$, the equation~\eqref{eq:converge_b} becomes
\begin{equation*}
  \Ebb\paren*{e^{\ic t \frac{c_B(Z_n,\xi)-n\ell}{\sqrt{n}}}}
  = e^{-\ic t\sqrt{n}\ell} e^{\ic a \sqrt{n} t - b t^2/2 + o(1)} + o(1)
  \to e^{-b t^2/2}.
\end{equation*}

We recall a version of L\'evy's theorem: if a sequence of real random variables
$X_n$ satisfies $\Ebb(e^{\ic t X_n}) \to \phi(t)$ where $\phi$ is a
continuous function, then $X_n$ converges in distribution to a random
variable $X$, and $\phi(t) = \Ebb(e^{\ic t X})$. Hence,
$(c_B(Z_n,\xi)-n\ell)/\sqrt{n}$ converges to a random variable, whose
characteristic function is $e^{-b t^2/2}$. This implies that $b\in
[0,\infty)$. Writing it as $\sigma^2 \geq 0$, we have proved that
$(c_B(Z_n,\xi)-n\ell)/\sqrt{n}$ converges to $\boN(0,\sigma^2)$ for any $\xi
\in \partial'_B \Gamma$.

Finally,
\begin{equation*}
  \frac{d(e,Z_n)-n\ell}{\sqrt{n}} = \frac{c_B(Z_n,\xi)-n\ell}{\sqrt{n}} + \frac{d(e,Z_n)- c_B(Z_n,\xi)}{\sqrt{n}}.
\end{equation*}
The first term on the right converges to $\boN(0,\sigma^2)$. The second term
on the right converges to $0$ in probability by~\eqref{eq:cBproba2}. Hence,
their sum also converges to $\boN(0,\sigma^2)$. This concludes the proof of
the central limit theorem for $d(e,Z_n)$.

\medskip

Let us show that $\sigma^2(\mu)$ depends analytically on $\mu$. One checks
easily that the map $(\mu,t)\mapsto L_{\mu,t}$ is analytic. Hence,
$(\mu,t)\mapsto \lambda_\mu(t)$ is also analytic, as simple isolated
eigenvalues depend analytically on the operator, see~\cite{kato_pe}. As a
consequence, $\mu\mapsto \lambda_\mu'(0)=\partial\lambda_\mu(t)/\partial t
\restr_{t=0}$ is analytic, and so is $\mu \mapsto \lambda_\mu''(0)$. Since
$\ell(\mu)=a=-\ic \lambda_\mu'(0)$ and
$\sigma^2(\mu)=b=-a^2-\lambda_\mu''(0)$, we recover simultaneously the
analyticity of the escape rate stated in Theorem~\ref{thm:main}, and the
analyticity of the variance in Theorem~\ref{thm:CLT}.

\medskip

It remains to show that the variance $\sigma^2(\mu)$ is nonzero. For this, we
need to identify $b$, by a spectral computation. Let $\nu$ denote the unique
stationary measure of the random walk on $\partial'_B \Gamma$, i.e., the
fixed point of the operator $(L_\mu)^*$. Define $u_t =\Pi_{1,t} 1/\int
\Pi_{1,t} 1\dd\nu$, it is an eigenfunction for the eigenvalue $\lambda(t)$ of
$L_{\mu,t}$, normalized by $\int u_t \dd\nu=1$. Note that $\int \Pi_{1,0}1
\dd\nu = \int 1 \dd\nu=1$, so that the denominator in the definition of $u_t$
does not vanish for small enough $t$. As $u_t$ depends analytically on $t$,
we may write $u_t=u_0 + \ic t v + O(t^2)$, with $u_0=1$ and $\int v\dd\nu=0$
as $\int u_t \dd\nu=1$ for all $t$. Moreover, $L_{\mu,t} u_t = \lambda(t)
u_t$. Keeping only the terms of order $\leq 1$ in this equation and using
$\lambda(t)=1+\ic t a + O(t^2)$, we get for all $\xi\in \partial_B\Gamma$
\begin{equation*}
  \sum_g \mu(g) e^{\ic t c_B(g,\xi)} (1+\ic t v(g\xi)) = (1+\ic t a)(1+\ic t v(\xi))+O(t^2),
\end{equation*}
i.e.,
\begin{equation*}
  1 + \ic t\sum_g \mu(g) c_B(g,\xi) + \ic t \sum_g \mu(g) v(g\xi) = 1+\ic t a + \ic t v(\xi)+O(t^2).
\end{equation*}
Looking at the coefficient of $\ic t$, we get
\begin{equation}
\label{eq:coboundary}
  \sum_g \mu(g) (c_B(g,\xi) + v(g\xi)-v(\xi)-a) = 0.
\end{equation}
Integrating this equation with respect to $\nu$ and using its stationarity,
we recover the equality
\begin{equation*}
  a=\int c_B(g,\xi) \dd\mu(g) \dd\nu(\xi) = \ell,
\end{equation*}
that we proved before by a probabilistic argument.

Define a new cocycle
\begin{equation*}
  \tilde c(g,\xi) = c_B(g,\xi) + v(g\xi)-v(\xi) -\ell.
\end{equation*}
It satisfies $\sum \mu(g) \tilde c(g,\xi)=0$ for all $\xi$
by~\eqref{eq:coboundary}. Moreover, $c_B(Z_n,\xi)-n\ell - \tilde c(Z_n,\xi) =
v(\xi) - v(Z_n \xi)$ is uniformly bounded, hence it is equivalent to have the
central limit theorem for $c_B$ or for $\tilde c$, and the asymptotic
variances coincide.

Let us use $\tilde c$ instead of $c_B$, to define a new operator $\tilde L_t$
by $\tilde L_t u(\xi) = \sum_g \mu(g) e^{\ic t \tilde c(g,\xi)} u(g\xi)$. All
the above discussion applies to this operator. We get in particular a new
eigenvalue $\tilde \lambda(t)$, a new eigenfunction $\tilde u_t$ which can be
expanded as $1+\ic t \tilde v + O(t^2)$ with $\int \tilde v \dd\nu = 0$, a
new value of the derivative $\tilde a$ of $\tilde\lambda$ at $0$. The
equation~\eqref{eq:coboundary} becomes
\begin{equation}
\label{eq:ilmwuxcv}
  \sum_g \mu(g) (\tilde c(g,\xi) + \tilde v(g\xi)-\tilde v(\xi) -\tilde a) = 0.
\end{equation}
By construction, $\sum \mu(g) \tilde c(g,\xi)=0$ for all $g$. Integrating the
above equation with respect to $\nu$, we get $\tilde a =0$.
Then~\eqref{eq:ilmwuxcv} yields $\tilde v=L_0 \tilde v$. By
Proposition~\ref{prop:spectral_description}, $\tilde v$ is constant. As its
integral is zero, $\tilde v=0$. Finally, $\tilde u_t=1+O(t^2)$.

We can now compute the second term in the expansion of $\tilde\lambda(t)$. We
have
\begin{align*}
  \tilde \lambda(t) & = \int \tilde\lambda(t) \tilde u_t \dd\nu
  =\int \tilde L_t \tilde u_t \dd\nu
  =\int \tilde L_t (\tilde u_t-1)\dd\nu + \int \tilde L_t 1 \dd\nu
  \\&
  =\int (\tilde L_t-L_0) (\tilde u_t-1) \dd\nu + \int \tilde L_t 1 \dd\nu.
\end{align*}
The first term is $O(t^3)$ as $\tilde u_t-1=O(t^2)$ and $\tilde L_t
-L_0=O(t)$. The second term is equal to
\begin{equation*}
  \int e^{\ic t \tilde c(g,\xi)} \dd\mu(g) \dd\nu(\xi)
  =1+\ic t \int \tilde c(g,\xi) \dd\mu(g) \dd\nu(\xi) -\frac{t^2}{2} \int
\tilde c(g,\xi)^2 \dd\mu(g)\dd\nu(\xi) + O(t^3).
\end{equation*}
Since the variance in the central limit theorem is given by
$-\tilde\lambda''(0)$, we obtain
\begin{equation*}
  \sigma^2 = \sum_g \mu(g) \int \tilde c(g,\xi)^2 \dd\nu(\xi)\geq 0.
\end{equation*}
If it vanishes, then all quantities $\tilde c(g,\xi)$ are zero for $g$ in the
support of $\mu$. Using the definition of $\tilde c$ and the cocycle equation
for $c_B$, this gives
\begin{equation*}
  c_B(Z_n,\xi) = n\ell + v(\xi)-v(Z_n \xi)
\end{equation*}
for any $Z_n$ in the support of $\mu^{*n}$. Take $n$ with $\mu^{*n}(e)>0$,
by~\eqref{eq:exists_good_N}. Applying the above equation to $Z_{n}=e$, we get
$0=n\ell$, a contradiction with the positivity of $\ell$ that follows from
the non-amenability of $\Gamma$. Thus, $\sigma^2 > 0$.\qed

\medskip

\begin{rmk}
Several of the above arguments could be replaced by martingale arguments, as
in~\cite{benoist_quint_TCL_hyperbolic}, but we have opted for a
self-contained spectral treatment. Note that many objects that appear in this
proof are also present in the martingale argument. For instance, the key step
in the martingale argument is to construct a function $v$ such
that~\eqref{eq:coboundary} holds. This function $v$ (a solution of the
Poisson equation) appears naturally in the spectral framework, as the
derivative of $t\mapsto u_t$.
\end{rmk}

There is nothing special about the Busemann cocycle in the above proof: the
next theorem is proved exactly in the same way.

\begin{thm}
Let $\Gamma$ be a nonelementary hyperbolic group with a word distance,
endowed with an admissible probability measure $\mu$ having an exponential
moment. Let $c: \Gamma \times \partial'_B \Gamma \to \R$ be a cocycle on a
minimal subset $\partial'_B \Gamma$ of the Busemann boundary $\partial_B
\Gamma$. Assume that the cocycle is H\"older continuous, i.e., for every $g
\in \Gamma$, the map $\xi \mapsto c(g, \xi)$ is H\"older continuous on
$\partial'_B \Gamma$.

Let $Z_n$ denote the right random walk on $\Gamma$ driven by $\mu$. Then
there exists $\ell \in \R$ such that, for any $\xi \in \partial'_B \Gamma$,
the quantity $c(Z_n, \xi)/n$ converges almost surely to $\ell$. Moreover,
there exists $\sigma^2 \geq 0$ such that
\begin{equation*}
 \frac{c(Z_n, \xi) - n\ell}{\sqrt{n}} \to \boN(0,\sigma^2).
\end{equation*}
Additionally, $\sigma^2=0$ if and only if there exists a H\"older continuous
function $v$ on $\partial'_B \Gamma$ such that $c(g,\xi) = v(g\xi) - v(\xi)$
for all $g\in \Gamma$ and all $\xi \in \partial'_B \Gamma$ (and in this case
$\ell=0$ also). This is also equivalent to the uniform boundedness of $c$.

Finally, both $\ell$ and $\sigma$ depend analytically on $\mu$.
\end{thm}

\bibliography{biblio}
\bibliographystyle{amsalpha}


\begin{dajauthors}
\begin{authorinfo}[sg]
  S\'ebastien Gou\"ezel\\
  Laboratoire Jean Leray, CNRS UMR 6629\\
  Universit\'e de Nantes, 2 rue de la Houssini\`ere\\
  44322 Nantes, France\\
  \url{sebastien.gouezel@univ-nantes.fr}
\end{authorinfo}
\end{dajauthors}

\end{document}